\documentclass[11pt]{amsart}
\usepackage{amsfonts,graphics,amsmath,amsthm,amscd,amssymb,enumerate,latexsym,multicol,mathrsfs}
\usepackage{graphicx, hhline}
\usepackage[all]{xy}
\usepackage[colorlinks=true,linkcolor=blue,anchorcolor=black,citecolor=red]{hyperref}
\usepackage{fancyhdr}
\usepackage{url}
\usepackage[a4paper,margin=28mm]{geometry}
\usepackage[T1]{fontenc}
\usepackage{lmodern}

\title[P-trivial MMP and Zariski decompositions]{$P$-trivial MMP, Zariski decompositions and minimal models for generalised pairs}
\author{Zhengyu Hu}
\date{2026/11/15}

\address{Mathematical Sciences Research Center, Chongqing University of Technology, No.69 Hongguang Avenue, Chongqing, 400054, China}
\email{zhengyuhu16@gmail.com}

\newcommand{\Exc}[0]{{\operatorname{Exc}}}
\newcommand{\Supp}[0]{{\operatorname{Supp}}}

\DeclareMathOperator{\NEbar}{\overline{\mathrm{NE}}}

\DeclareMathOperator{\mult}{mult}

\DeclareMathOperator{\ex}{Ex}

\newtheorem{thm}{Theorem}[section]

\newtheorem{lem}[thm]{Lemma}
\newtheorem{cor}[thm]{Corollary}
\newtheorem{prop}[thm]{Proposition}
\newtheorem{conj}[thm]{Conjecture}

\theoremstyle{definition}
\newtheorem{defn}[thm]{Definition}
\newtheorem{rem}[thm]{Remark}

\newtheorem{exa}[thm]{Example}

\newtheorem*{claim*}{Claim}

\newcommand{\K}{\mathbb K}
\newcommand{\PP}{\mathbb P}

\newcommand{\Q}{\mathbb Q}
\newcommand{\R}{\mathbb R}
\newcommand{\Z}{\mathbb Z}
\newcommand{\bir}{\dashrightarrow}
\newcommand{\rddown}[1]{\left\lfloor{#1}\right\rfloor}

\begin{document}
	
	\begin{abstract}
		We develop an inductive method for generalised pairs (g-pairs) without the NQC assumption. It is based on $P$-trivial MMPs, whose steps are trivial with respect to a given nef divisor $P$. For klt g-pairs, if $P$ is numerically equivalent to an effective divisor, we perturb the boundary so that any MMP on $K_X+B+M+\alpha P$ is $P$-trivial for $\alpha\gg0$. For lc g-pairs, the components with coefficient one cannot be perturbed, so we use induction on lc centres. If $K_X+B+M-P$ is a $\Q$-divisor, then local Cartier indices control the coefficients under divisorial adjunction, and we prove a special termination. This includes the case $P=0$, where $K_X+B+M$ is a $\Q$-divisor.
		
		As applications, we prove that a projective lc g-pair $(X,B+M)$ of dimension $3$ has a minimal model if $K_X+B+M$ birationally has a Nakayama--Zariski decomposition with nef positive part and either $K_X+B+M\equiv D\geq0$ or $M\equiv D\geq0$. We also prove this for klt g-pairs in arbitrary dimension. In dimension $4$, the last condition can be replaced by the log abundance of the nef data. For usual pairs, this is a theorem of Birkar and the author.
	\end{abstract}
	
	\maketitle

	\noindent\textbf{Use of AI tools.} The first version of this paper (arXiv:2501.07551v1, submitted on 13 January 2025) was written by the author in 2024, before any AI tools were used for this paper. All the theorems, propositions, lemmas and corollaries proved in this paper were obtained by the author in that version. 
	
	During subsequent revisions, the author used AI tools only for the following:
	\begin{enumerate}
		\item adding the illustrative Examples~\ref{exa-P-nef-curves}, \ref{exa-P-minimal-no-threshold}, \ref{exa-boundary-no-degeneration} and \ref{exa-non-nqc-application}, which are not in the first version;
		\item assisting with revising the introduction and improving the English;
		\item identifying an error in the curve-based criterion for $P$-nefness in the first version (see Definition~\ref{defn-P-nef}).
	\end{enumerate}
	All the mathematics in this paper remains entirely the author's own contribution. The author takes full responsibility for the mathematical content and final text.

	\tableofcontents
	
	\section{Introduction}\label{sec1}
	
	We work over an algebraically closed uncountable field $k$ of characteristic zero. All varieties are
	quasi-projective over $k$ and a divisor refers to an $\R$-Weil divisor unless stated otherwise.
	
	Generalised pairs (g-pairs for short) naturally appear in the canonical bundle formula and play a crucial role in recent birational geometry \cite{birkarzhang, Bir21b}. Most results on generalised pairs assume that the nef data are NQC, that is, a finite non-negative real linear combination of nef $\Q$-Cartier divisors. This allows us to use the techniques for usual pairs. For example, Zariski decompositions and minimal models of NQC generalised pairs are studied in \cite{hanli, LX23a, txie}. Generalised pairs also arise in the minimal model program for K\"ahler varieties and for transcendental classes. For example, Das, Hacon and Y\'a\~nez developed the MMP for generalised pairs on compact K\"ahler threefolds \cite{DHY23} (see also \cite{DH23, HP24}), and Das and Hacon proved that a projective generalised klt pair $(X,B+\beta)$ with $B+\beta$ big has a good minimal model or a Mori fibre space \cite{DH24}. Moreover, \cite{chlx} proved the cone and contraction theorems for arbitrary lc g-pairs by the theory of algebraically integrable foliations. 

	Without the NQC assumption, the numerical nonvanishing fails in general \cite{hanliu}, and we do not know whether the coefficients under adjunction satisfy the DCC. Hence the usual special termination does not apply, and even in dimension three, it is not known whether minimal models exist for arbitrary g-pairs. In this paper, we develop an inductive method under a rationality condition instead of the NQC assumption. More precisely, given a nef divisor $P$, we say that an MMP is $P$-trivial if each of its steps is trivial with respect to $P$. Our main result is a special termination for $P$-trivial MMPs when $K_X+B+M-P$ is a $\Q$-divisor (Theorem \ref{main-thm-sp-term}). 
	
	Conjecture \ref{conj-nz-decomp} below is the starting point of this paper, but our method finally turns out to go beyond it (see Remark \ref{rem-g-mmp}). Given a projective lc g-pair $(X,B+M)$, we say that $K_X+B+M$ birationally has a Nakayama--Zariski decomposition with nef positive part if there is a birational model $\pi: X' \to X$ such that $P_\sigma(\pi^*(K_X+B+M))$ is nef. If $(X,B+M)$ has a minimal model, then $K_X+B+M$ birationally has a Nakayama--Zariski decomposition with nef positive part. It is natural to ask whether the converse is true.

	\begin{conj}\label{conj-nz-decomp}
		Let $(X,B+M)$ be a projective lc g-pair. Assume $K_X+B+M$ birationally has a Nakayama--Zariski decomposition with nef positive part. Then $(X,B+M)$ has a minimal model.
	\end{conj}

	For usual pairs, Conjecture \ref{conj-nz-decomp} was proved by Birkar and the author \cite{bhzariski}. The main technique is to run an MMP on $K_X+B+\alpha P$ for $\alpha\gg0$, where $P$ is the nef positive part. This MMP is $P$-trivial if the set $\{P\cdot C>0\}$ is bounded away from zero, where $C$ runs over extremal curves. Then it contracts the negative part and terminates with a minimal model. Here the bound follows from the length of extremal rays (see Lemma \ref{lem-P-trivial}). Related results for NQC g-pairs appear in \cite[Theorem 5.18]{txie}, where both the nef data and $P$ are NQC. If $P$ is NQC, then any MMP on $K_X+B+M+\alpha P$ is $P$-trivial for $\alpha\gg0$ (cf. \cite[Lemma 3.22]{hanli}), and if the nef data has an NQC part, see Remark \ref{rem-nqc-extension}. The following example is due to Han and Liu \cite[Example 1.3]{hanliu}, based on a construction of van Dobben de Bruyn \cite{vddb}.

	\begin{exa}\label{exa: nonvanishing fail}
		
		Let $E$ be a non-CM elliptic curve and $X=E\times E$. For a point $p\in E$, let $F_1=\{p\}\times E$, $F_2=E\times\{p\}$ and $\Delta$ be the diagonal. The key observation is that the boundary of the nef cone contains irrational rays. In fact, by calculation the divisor
		$$
		M=F_1+\sqrt{2}F_2+(\sqrt{2}-2)\Delta
		$$
		is nef with $M^2=0$, and it spans an irrational extremal ray of the nef cone. Suppose that $M\equiv D\geq0$. Then every prime component of $D$ is nef, and by extremality their numerical classes lie on this ray. This is impossible because they are nonzero integral classes. So $(X,0+M)$ is a klt g-pair with nef log canonical divisor, but the numerical nonvanishing fails. We also remark that the positive intersections $M\cdot C$, where $C\subset X$ runs over curves, are not bounded away from zero \cite{vddb} (see also \cite[Example 3.19]{hanli}).
	\end{exa}

	Note that in the above example, all curves are $K_X$-trivial, so an MMP does not contract them. However, the same phenomenon happens for negative extremal curves. In Example \ref{exa-P-minimal-no-threshold}, $T$ is a rational surface and $Q$ is a nef divisor which is not NQC. There are $K_T$-negative extremal curves $F_m$ such that $Q\cdot F_m>0$ and $Q\cdot F_m\to 0$. Moreover, $Q$ is positive on every $K_T$-negative extremal curve, and $K_T+Q+\alpha Q$ is not nef for any $\alpha\geq0$. So for the g-pair $(T,0+Q)$ and $P=Q$, the first step of an MMP on $K_T+Q+\alpha P$ is not $P$-trivial, however large $\alpha$ is. Therefore, without the NQC assumption, the method of \cite{bhzariski} cannot be applied directly.

	\noindent\textbf{The klt case.}
	The basic case is the following. Let $(X/Z,B+M)$ be a $\Q$-factorial dlt g-pair and $P$ be a nef$/Z$ divisor such that $\Supp\{P\}_{ir}\subseteq\Supp\{B\}$, where $\{P\}_{ir}$ denotes the fractional part of $P$ with irrational coefficients. Then any MMP on $K_X+B+M+\alpha P$ is $P$-trivial for $\alpha\gg0$ (Lemma \ref{lem-P-trivial}). The idea is to approximate $P$ by rational divisors and absorb the perturbations into the boundary. We intuitively explain why numerical effectivity plays a crucial role here. If $P\equiv D\geq0$, then because MMP is a numerical process, we can assume $P\geq0$. If moreover $\Supp P$ contains no lc centres, then $(X,(B+\epsilon P)+M)$ is still dlt for $0<\epsilon\ll1$, and we can apply the above lemma (see Theorem \ref{thm-klt}). With the above perturbation, we are able to establish Conjecture \ref{conj-nz-decomp} for klt g-pairs if either $M$ or $K_X+B+M$ is numerically equivalent to an effective divisor.

	\begin{thm}[\text{=Corollary \ref{cor-klt'}}]\label{main-klt'}
		Let $(X,B+M)$ be a projective klt g-pair. Assume $K_X+B+M$ birationally has a Nakayama--Zariski decomposition with nef positive part. Suppose either $M \equiv D \ge 0$ or $K_X+B+M \equiv D \ge 0$. Then $(X,B+M)$ has a minimal model.
	\end{thm}

	\noindent\textbf{The lc case.}
	For lc g-pairs, the components of $\rddown{B}$ have coefficient one and we cannot perturb them. So we need induction on lc centres. We run a $P$-trivial MMP with scaling of a birationally ample divisor $A$, that is, we successively run finite MMPs on
	$$
	L+\lambda_i A+\alpha_i P,\qquad \lambda_i \downarrow 0,
	$$
	where $L=K_X+B+M$ and $\alpha_i$ may increase as $\lambda_i$ decreases (Definition \ref{defn-P-mmp}). As we mentioned above, the coefficients under adjunction to lc centres are not known to satisfy the DCC. However, ${}^*P$ does not change along a $P$-trivial MMP, so the discrepancies of $(X^i,B^i+M^i)$ differ from those of $(X,B+M)$ by the coefficients of ${}^*N-{}^*N^i$, where $N=L-P$. If $N$ is a $\Q$-divisor, we bound their denominators by local Cartier indices, and the coefficients under adjunction belong to finite sets (Theorem \ref{thm-finite-coeff}). Moreover, we have the same bound for the restriction of $N$ to lc centres. As a consequence, following Fujino's argument \cite{fujino-sp}, we obtain the following special termination. It is the main result of this paper.

	\begin{thm}[\text{A special termination, =Theorem \ref{thm-sp-term}}]\label{main-thm-sp-term}
		Let $(X/Z,B+M)$ be a $\Q$-factorial dlt g-pair and $K_X+B+M=P+N$ where $P$ is a nef$/Z$ divisor and $N$ is a $\Q$-divisor. Given a $P$-trivial MMP on $K_X+B+M$ with scaling of some divisor $A$, let $S$ be an lc centre of $(X,B+M)$. Suppose that the MMP does not contract $S$ and the induced map $S^i\bir S^{i+1}$ is an isomorphism on $\rddown{B_{S^i}}$ for $i \gg 0$. Then after discarding finitely many steps, these maps lift to small $\Q$-factorialisations $\widetilde S^i\bir\widetilde S^{i+1}$ forming a $P^i|_{\widetilde S^i}$-trivial MMP with scaling of $A^i|_{\widetilde S^i}$.
	\end{thm}

	Roughly speaking, a $P$-trivial MMP induces a $P$-trivial MMP of the same kind on each lc centre. To the best of my knowledge, this is the first inductive approach to g-pairs without the NQC assumption. Note that in the setting of Conjecture \ref{conj-nz-decomp}, we can assume that $N$ is a $\Q$-divisor after passing to a log resolution and perturbing the boundary slightly (see the construction before Theorem \ref{thm-dlt'}).

	\begin{rem}\label{rem-g-mmp}
		The techniques of $P$-trivial MMPs and special termination in this paper also lead to a short and independent proof of the basic theory of the MMP for lc generalised pairs, without the NQC or $\Q$-factoriality assumption, including the cone and contraction theorems, the existence of flips and theorems of BCHM type. These proofs use \cite{bchm} as the only major theorem on the MMP. They will appear in a forthcoming paper.
	\end{rem}

	Together with termination results in lower dimensions, Theorem \ref{main-thm-sp-term} implies termination near $\rddown B$. It is natural to ask when a $P$-trivial MMP is an MMP on $L+\alpha P$ for a fixed $\alpha$, with the same scaling numbers. In this case, we say that the MMP degenerates (Definition \ref{defn-degen}). If $P=P_\sigma(L)$, then degeneration implies the existence of a minimal model (Lemma \ref{lem-degen}). In Theorem \ref{thm-degen}(3), we prove that, under certain numerical hypotheses, termination near the lc centres and degenerations on the lc centres imply degeneration. Note that the second requirement does not follow from the first one (see Example \ref{exa-boundary-no-degeneration}).

	\noindent\textbf{Applications.}
	In dimension three, we obtain the following theorem.

	\begin{thm}[\text{=Corollary \ref{cor-3-dim'}}]\label{main-3-dim'}
		Let $(X,B+M)$ be a projective lc g-pair of dimension $3$. Assume $K_X+B+M$ birationally has a Nakayama--Zariski decomposition with nef positive part. Suppose either $K_X+B+M \equiv D \ge 0$ or $M \equiv D \ge 0$.
		Then $(X,B+M)$ has a minimal model.
	\end{thm}

	Even the three-dimensional case needs Theorem \ref{main-thm-sp-term}, since without it we cannot pass from surfaces to threefolds. Example \ref{exa-non-nqc-application} illustrates this result for a g-pair whose nef data and positive part are not NQC. If the nef data is log abundant, we can prove the following four-dimensional result.

	Given a log smooth sub-dlt pair $(X',B')$ and a divisor $D'$, we say $D'$ is \emph{log effective (resp. log abundant)} if for any stratum $S'$ of $(X',B'^{=1})$ (including $X'$), we have $D'|_{S'} \equiv E' \ge 0$ for some $E'$ (resp. $D'|_{S'}$ is abundant).

	\begin{thm}[\text{=Corollary \ref{cor-dlt-abund}}]\label{main-dlt-abund}
		Let $(X,B+M)$ be a projective lc g-pair of dimension $4$ with data $M'$. Assume $K_X+B+M$ birationally has a Nakayama--Zariski decomposition with nef positive part. Suppose $M'$ is log abundant.
		Then $(X,B+M)$ has a minimal model.
	\end{thm}

	\noindent\textbf{Higher-dimensional and related results.}
	We also prove higher-dimensional versions of the lc results, namely Theorems \ref{thm-dlt'} and \ref{thm-dlt-abund'}. Under log effectivity of the nef data or the positive part, the former reduces the problem in dimension $d$ to the restricted termination conjecture \ref{conj-term} in dimensions at most $d-1$. Under log abundance of the nef data, the latter assumes termination of MMPs for usual dlt pairs in dimensions at most $d-1$.
	
	Without a lower-dimensional termination assumption, we obtain weak minimal models. This notion, recalled in Definition \ref{defn-wmm}, comes from \cite{hu} for usual pairs, where it was used for a special MMP. It is weaker than that of a minimal model. Their existence is equivalent for NQC g-pairs (Proposition \ref{prop-wmm}), but this is not known for arbitrary g-pairs.
	
	\begin{thm}[\text{=Theorem \ref{thm-dlt''}}]\label{main-thm-dlt''}
		Let $(X,B+M)$ be a projective lc g-pair with data $M'$. Assume $K_X+B+M$ birationally has a Nakayama--Zariski decomposition with nef positive part. Suppose either $M'$ or the positive part is log effective. Then $(X,B+M)$ has a weak minimal model.
	\end{thm}

	\noindent\textbf{Acknowledgement.} The work is supported by Overseas High-Level Young Talent Recruitment Programs and partially supported by Chongqing Natural Science Foundation Innovation and Development Joint Fund CSTB2023NSCQ-LZX0031. The author would like to thank Professors Caucher Birkar, Chen Jiang, Jihao Liu for discussions and comments. The work was carried out when the author was visiting NCTS, Taipei. He thanks Professor Jungkai Chen for his hospitality.\\

	\section{Preliminaries}
	
	Let $\mathbb{K}$ denote $\Q$ or $\R$.\\
	\noindent \textbf{B-divisors.}
	Let $X$ be a variety. A b-divisor $\mathbf{D}$ of $X$ is an element of $\lim\limits_{\longleftarrow}\mathrm{Div}(X')$, where $X'$ runs over all birational models of $X$. More explicitly, $\mathbf{D}$ is given by a family
	$\{\mathbf{D}_{X'}\}$ of divisors indexed by all birational models $X'$ of $X$
	such that $\mu_*(\mathbf{D}_{X''}) = \mathbf{D}_{X'}$ if $\mu: X'' \to X'$ is a birational contraction.
	
	A \emph{$\K$-Cartier b-divisor} $\mathbf{M}$ is defined by
	a projective birational morphism
	$X' \to X$ from a normal variety $X'$ and a $\K$-Cartier divisor $M'$ on $X'$ such that $\mathbf{M}_{X''}=\mu^*M'$ for any birational model $\mu: X'' \to X'$. In this case we say that $\mathbf{M}$ \emph{descends} to $X'$ or $M'$ \emph{represents} $\mathbf{M}$, and denote it by $^*M'$.
	
	Given a $\K$-Cartier b-divisor $\mathbf{M}$ on $X$ represented by $M'$ and a surjective projective morphism $f: Y \to X$, we define the pull-back of $\mathbf{M}$ as the $\K$-Cartier b-divisor $f^*\mathbf{M}$ represented by $f'^*M'$ where $f': Y' \to X'$ is induced by $f$. Similarly, given a prime divisor $S$ of $X$, we define $\mathbf{M}|_S$ to be the b-divisor represented by $M'|_{S'}$ where $S'$ is the birational transform of $S$ on $X'$. Note that the definition is independent of the choice of $X'$.
	
	A $\K$-Cartier b-divisor represented by $M'$ is \emph{nef (resp. nef and abundant)} if $M'$ is
	nef (resp. nef and abundant). \\
	
	\noindent \textbf{Generalised pairs.}
	For the basic theory of generalised pairs we refer to \cite[Section 4]{birkarzhang}.
	
	A \emph{generalised sub-pair} (\emph{g-sub-pair} for short) $(X/Z,B+M)$ with data $M'$ consists of
	\begin{itemize}
		\item a normal variety $X$ equipped with a projective
		morphism $X\to Z$,
		
		\item a divisor $B$ on $X$, and
		
		\item an $\R$-Cartier b-divisor $^*M'$ represented
		by some birational model $X' \overset{\phi}\to X$ and a nef divisor
		$M'$ on $X'$ such that $K_{X}+B+M$ is $\R$-Cartier,
		where $M := (^*M')_X$.
	\end{itemize}
	A g-sub-pair is a \emph{g-pair} if $B$ is effective. We omit the data $M'$ when they are clear.
	
	We say a g-sub-pair $(X'',B'' +M'')$, where $K_{X''}+B''+M''$ is the pull-back of $K_X+B+M$, is \emph{log smooth} if $(X'',B'')$ is log smooth and $^*M''=^*M'$. In this case, we call $(X'',B''+M'')$ a \emph{log resolution} of $(X,B+M)$.
	When $Z$ is a point, we omit it and say that the pair is projective. Given a prime divisor $D$ over $X$, we define the discrepancy $a(D,X,B+M)= 1- \mu_D B''$ where $\mu_D$ denotes the multiplicity on the log resolution $X'' \to X$.
	
	We say a g-pair $(X/Z,B+M)$ is
	\emph{lc} (resp. \emph{klt})
	if for each $D$ the discrepancy $a(D,X,B+M)$ is $\ge 0$ (resp. $>0$).
	An lc g-pair $(X/Z,B+M)$ with data $M'$ is \emph{dlt} if
	there exists a log resolution $(X'/Z,B'+M') \to X$ such that the discrepancy $a(E,X,B+M) >0$ for every exceptional$/X$ divisor on $X'$. If in addition $\rddown{B}$ is normal, we say the pair is \emph{plt}.
	
	Our definition of dlt g-pairs coincides with that in \cite{birkar-BAB}. The following lemma proves the implication from our definition.
	\begin{lem}
		Given a dlt g-pair $(X,B+M)$ with data $M'$ on $\phi:X' \to X$, let $\eta$ be the generic point of an lc centre of $(X, B+M)$. Suppose $\phi$ is a log resolution such that $a(E,X,B+M) >0$ for all exceptional divisors on $X'$. Then $\phi$ is an isomorphism over $\eta$. In particular, $(X, B)$ is log smooth near $\eta$ and $M'' = \phi''^*M$ holds for any log resolution $\phi'':X'' \to X$ over a neighbourhood of $\eta$.
	\end{lem}
	\begin{proof}
		Let $G:=\mathrm{Ex}(\phi)$. By assumption, $(X',B'+\epsilon G)$ is sub-dlt for $\epsilon \ll 1$. For any divisor $E$ with $c_X E \in Z:=\phi(G)$, we have $a(E,X,B+M) > a(E,X',B'+\epsilon G) > 0$. So $\eta \in X \setminus Z$, and hence $(X, B)$ is log smooth and $X' \bir X$ is an isomorphism near $\eta$.
	\end{proof}
	The dlt and plt conditions are preserved under an MMP.\\
	
	\noindent \textbf{Minimal models.}
	An lc g-pair $(Y/Z,B_Y+M_Y)$ with data $M_{Y'}$ is a \emph{log birational model} of an lc g-pair $(X/Z,B+M)$ with data $M'$ if there is a birational map
	$\phi\colon X\bir Y$, $B_Y=B^\sim+E$ where $B^\sim$ is the birational transform of $B$ and
	$E$ is the reduced exceptional divisor of $\phi^{-1}$ and $^*M_{Y'}=^*M'$.
	
	A log birational model $(X'/Z,B'+M')$ is a \emph{log smooth model} of $(X/Z,B+M)$ if it is log smooth with data $M'$.
	
	A log birational model $(Y/Z,B_Y+M_Y)$ is a \emph{weak lc model} of $(X/Z,B+M)$ if
	
	$\bullet$ $K_Y+B_Y+M_Y$ is nef$/Z$, and
	
	$\bullet$ for any prime divisor $D$ on $X$ which is exceptional/$Y$, we have
	$$
	a(D,X,B+M)\le a(D,Y,B_Y+M_Y).
	$$
	
	A weak lc model $(Y/Z,B_Y+M_Y)$ is a \emph{minimal model} of $(X/Z,B+M)$ if
	
	$\bullet$ $Y$ is $\Q$-factorial,
	
	$\bullet$ the above inequality on log discrepancies is strict.\\
	
	\noindent \textbf{Good minimal models.}
	The \emph{canonical model} of $(X/Z,B+M)$ is the ample model of $K_X+B+M$ over $Z$ (\cite[Definition 3.6.5]{bchm}).
	
	A minimal model $(Y/Z, B_Y+M_Y)$ is \emph{good} (resp. \emph{numerically good}) if $K_Y + B_Y +M_Y$ is semi-ample$/Z$ (resp. the numerical class of $K_Y + B_Y +M_Y$ contains a semi-ample divisor). If the model is good, $K_Y + B_Y +M_Y$ defines a contraction $g:Y \to W$ such that $K_Y + B_Y +M_Y=g^*A_W$ for some ample$/Z$ divisor $A_W$. In this case, $W$ is the canonical model of $(X/Z,B+M)$.  \\
	
	For the definitions of invariant Iitaka dimensions and numerical dimensions we refer to \cite[Section 2]{hashizumehu}.
	
	\begin{defn}[Log abundance]
		Let $f\colon X\to Z$ be a projective morphism from a normal variety to a variety, and $D$ be an $\mathbb{R}$-Cartier divisor on $X$.
		We say that $D$ is {\em abundant over} $Z$ if the equality $\kappa_{\iota}(X/Z,D)=\kappa_{\sigma}(X/Z,D)$ holds.
		
		Given a dlt g-pair $(X/Z,B+M)$, we say that $D$ is {\em log abundant$/Z$} with respect to $(X,B+M)$ if $D$ is abundant over $Z$ and for any lc centre $S$ of $(X,B+M)$, the restriction $D|_{S}$ is abundant over $Z$.
	\end{defn}
	
	The next lemma is \cite[V. 2.3 Lemma]{nakayama}. The original statement asserts $\pi^*D \sim_\Q g^*B$, which does not necessarily hold unless $D$ is effective. For a proof using invariant Iitaka fibrations, see \cite{hu23} and \cite[Lemma 2.8]{hu2}.
	
	\begin{lem}[Nef and abundant divisor, \text{cf.\cite[V. 2.3 Lemma]{nakayama}}]\label{lem-nef-abundant-divisor}
		Let $f\colon X\to Z$ be a projective morphism from a normal variety to a variety, and let $D$ be a nef$/Z$ $\R$-Cartier divisor on $X$. Then the following conditions are equivalent:
		\begin{enumerate}
			\item $D$ is abundant over $Z$.
			
			\item There exist a birational model $\pi: X' \to X$, a surjective morphism $g : X' \to Y$ of smooth quasi-projective varieties over $Z$,
			and a nef and big$/Z$ divisor $B$ of $Y$ such that $\pi^*D \sim_\R g^*B$.
		\end{enumerate}
	\end{lem}
	
	\noindent \textbf{Nakayama--Zariski decompositions.}
	Nakayama \cite{nakayama} defined a decomposition $D=P_\sigma(D)+N_\sigma(D)$ for any pseudo-effective
	divisor $D$ on a smooth projective variety. We refer to this as the \emph{Nakayama--Zariski decomposition}.
	We call $P_\sigma$ the positive part and $N_\sigma$ the negative part. We can
	extend it to the singular case as follows.
	Let $X$ be a normal projective variety and $D$ be a pseudo-effective $\R$-Cartier divisor on $X$. We define $P_\sigma(D)$
	by taking a resolution $f\colon W\to X$ and letting $P_\sigma(D):=f_*P_\sigma(f^*D)$. For properties of Nakayama--Zariski decompositions and more details, we refer to \cite{nakayama} and \cite{bhzariski}.  \\
	
	\section{Minimal model theory for arbitrary g-pairs}
	
	\subsection{Basic results}
	We collect some basic results for g-pairs and show that our definition of dlt coincides with that in \cite{birkar-BAB}.
	
	We begin with two consequences of the negativity lemma.
	\begin{lem}[{\cite[Proposition 3.9]{hanli}, \cite[Lemma 2.18]{hu2}, cf.\cite[Theorem 3.4]{birkar-flip}}]\label{lem-exc-2}
		Let $(X/Z, B+M)$ be an lc g-pair such
		that $K_X + B+M \equiv E/Z$ with $E \ge 0$ very exceptional$/Z$ and that $X$ is $\Q$-factorial klt. Then any MMP$/Z$
		on $K_X + B+M$ with scaling of an ample$/Z$ divisor terminates with a minimal model $Y$ on
		which $K_Y + B_Y +M_Y \equiv E_Y = 0/Z$.
	\end{lem}
	
	\begin{lem}[{\cite[Proposition 3.9]{hanli}, \cite[Lemma 2.20]{hu2}, cf.\cite[Theorem 3.5]{birkar-flip}}]\label{lem-exc-3}
		Let $(X/Z, B+M)$ be an lc g-pair such that $X$ is $\Q$-factorial klt, $X \to Z$ is birational and $K_X + B+M \equiv Q = Q_+ - Q_-/Z$ where $Q_+$, $Q_- \ge 0$ have
		no common components and $Q_+$ is exceptional$/Z$. Then any MMP$/Z$ on
		$K_X + B+M$ with scaling of an ample$/Z$ divisor contracts $Q_+$ after finitely many steps.
	\end{lem}
	
	\begin{lem}\label{lem-gmm}
		Let $(X/Z,B+M)$ be an lc g-pair projective over $Z$. Suppose it has the canonical model and $K_X+B+M$ is abundant$/Z$. Then $(X/Z,B+M)$ has a good minimal model.
	\end{lem}
	\begin{proof}
		After replacing $X$, we can assume that $M$ is nef$/Z$ and $g: X \bir T$ is a morphism to the canonical model $T$.
		Let $h: Y \to T$ be a nonsingular birational model. We can assume that $f : X \bir Y$ is a morphism. By the definition of the canonical model, we have $0 \le D \sim_\R K_X+B+M \sim_\R f^*P_{Y} +N$ for some big and semi-ample$/Z$ divisor $P_Y$ and $N \ge 0$. Since $D$ is abundant, we have $ \kappa_\sigma(X/Y,D)=0$ (\cite[Inequality 3.3]{fujino-additivity}, cf.\cite[V, 4.1]{nakayama}). Indeed, passing to a very general fibre, we can assume that $Z$ is a point. Then
		$$
		\kappa_\sigma(X,D+\alpha f^*P_Y) \ge \kappa_\sigma(X/Y,D) + \kappa (Y,\alpha P_Y),
		$$
		where $\alpha>0$. Since $P_Y$ is big, $\kappa(Y,\alpha P_Y)=\dim Y=\dim T$, which forces $ \kappa_\sigma(X/Y,D)=0$ as $ \kappa_\sigma(X,D+\alpha f^*P_Y)= \kappa_\sigma(X,D)=\dim T$.
		By \cite[Theorem B.6]{hu2}, after replacing $(X,B+M)$ and $Y$, we can further assume that:
		\begin{itemize}
			\item $f:(X,\Delta) \to (Y,\Delta_Y)$ is toroidal and equidimensional from a quasi-smooth toroidal pair (hence $\Q$-factorial) to a smooth toroidal pair (hence log smooth),
			
			\item $B,M,D$ are supported by $\Delta$.
		\end{itemize}
		
		Let $F$ be a very general fibre of $f$. Since $\kappa_\sigma(F,D^h|_F)=0$, we have
		$$
		D^h|_F=N_\sigma(D^h|_F).
		$$
		Running an MMP$/Y$ with ample scaling, we contract the components of $D^h$ after finitely many steps (cf.\ the proof of \cite[Proposition 3.9]{hanli}). We can write $D^v\sim_\R E/Y$ with $E\geq0$ very exceptional$/Y$. Its transform remains very exceptional, so by Lemma \ref{lem-exc-2} there is a minimal model $(X',B'+M')$ with $D'\sim_\R0/Y$.
		
		By the definition of the canonical model, $D'  \sim_\R  f^*D_Y=f^*(P_Y+N_Y)$, where $f^*N_Y \ge 0$ is contained in the fixed part of $|D'/Z|_\R$ and $P_Y=h^*P_T$ for some ample divisor $P_T$. It follows that $N_Y \ge 0$ is contained in the fixed part of $|D_Y/Z|_\R$ and is exceptional$/T$. Hence $N':=f^*N_Y$ is very exceptional$/T$. By Lemma \ref{lem-exc-2}, an MMP$/T$ on $K_{X'}+B'+M'$ terminates with a model $X''$ on which the push-down $D'' \sim_\R 0/T$. Therefore $(X''/Z,B''+M'')$ is a good minimal model.
	\end{proof}
	
	\noindent \textbf{Boundedness of intersection numbers.}
	We record bounds for intersection numbers with extremal curves, which will be used to run the MMP for arbitrary g-pairs.
	\begin{defn}[Extremal curves of minimal length, \text{cf.\cite[Section 3]{hanli}}]\label{defn-ext-curve}
		Let $X \to Z$ be a projective morphism from a normal variety and $\Gamma$ be a curve on $X$. We say $\Gamma$ is \emph{extremal of minimal length} if
		\begin{itemize}
			\item $\Gamma$ generates an extremal ray of $\overline{NE}(X/Z)$; and
			
			\item $H\cdot \Gamma$ is minimal for all curves in the extremal ray, where $H$ is an ample$/Z$ Cartier divisor.
		\end{itemize}
	\end{defn}
	
	By the length of extremal rays, we have the following lemma.
	\begin{lem}[\text{cf.\cite[Proposition 3.17]{hanli}}]\label{lem-lower-bnd}
		Let $X \to Z$ be a projective morphism from a $\Q$-factorial klt variety and $\Gamma$ be an extremal curve of minimal length. Then for any boundary $B$ and data $M'$ such that $(X,B+M)$ is lc, we have $(K_X+B+M) \cdot \Gamma \ge -2 \dim X$.
	\end{lem}
	
	\begin{lem}[Boundedness of intersection numbers]\label{lem-bound-int-num}
		Let $(X/Z,B+M)$ be a $\Q$-factorial dlt g-pair and $D$ be a divisor. Suppose $(X,(B+tD)+M)$ is lc for some real number $t>0$. Then there is a number $l >0$ depending on $t$ and $\dim X$ such that, for any $(K_X+B+M)$-negative extremal curve $\Gamma$ of minimal length, $D\cdot \Gamma > - l$.
	\end{lem}
	\begin{proof}
		Since $(X,B+t D + M)$ is lc, we have $(K_X+B+t D +M) \cdot \Gamma \ge -2 \dim X$, hence $D\cdot \Gamma > -2 \dim X / t$.
	\end{proof}
	
	\begin{rem}\label{rem-int-num}
		With the notation above, if a birational contraction $X \bir  X^i$ is $(K_X+B+tD+M)$-non-positive (for example, an MMP on $K_X+B+tD+M$), then the lower bound $l$ is preserved for the birational transform $D^i$ of $D$.
	\end{rem}
	
	We immediately obtain the next lemma.
	
	\begin{lem}[\text{cf.\cite[Theorem 3.2]{bhzariski}, \cite[Lemma 3.22]{hanli}}]\label{lem-P-trivial}
		Let $(X/Z,B+M)$ be a $\Q$-factorial dlt g-pair and $P$ be a nef divisor. Suppose that $\Supp \{P \}_{ir} \subseteq \Supp \{B\} $, where $\{\cdot\}_{ir}$ denotes the irrational fractional part  (the fractional part with irrational coefficients). Then any MMP on $K_X+B+M +\alpha P$ is $P$-trivial for all $\alpha \gg 0$.
	\end{lem}
	\begin{proof}
		Let $P=\sum_{j} \alpha_j P_j$ be a convex combination of $\Q$-divisors such that the coefficients $\alpha_j$ are $\Q$-linearly independent, $\Supp(P_j-P)\subseteq\Supp\{P\}_{ir}$ and $\|P_j -P\| \ll 1$. Let $N:=K_X+B+M-P$ and $B_j:=B+P_j-P$. Then $(X,B_j+M)$ is dlt and $N=K_X+B_j+M-P_j$. So for any $(K_X+B+M)$-negative extremal curve $\Gamma$ of minimal length, $P_j \cdot \Gamma$ is bounded from below by Lemma \ref{lem-lower-bnd} and $m_j P_j \cdot \Gamma \in \Z$, where $m_j>0$ is an integer such that $m_jP_j$ is Cartier. Hence the set $\{P \cdot \Gamma >0\}$ is bounded away from zero. Pick $\alpha \gg 0$ so that $\inf \{P \cdot \Gamma >0\} > 2 \dim X/\alpha$, and hence a step of an MMP on $K_X+B+M+\alpha P$ is $P$-trivial.
		
		Let $(X,B+M) \bir (X^k,B^k+M^k)$ be an MMP on $(K_X+B+M)$, where $B^k,M^k$ are birational transforms. Inductively, suppose every step is $P^i$-trivial. Since each $P^i$-trivial step is also $P_j^i$-trivial, we deduce that $(X^k,B^k_j +M^k)$ are all dlt. Hence by Remark \ref{rem-int-num} the lower bounds of $P_j^k \cdot \Gamma$ depend only on $\dim X$, where $\Gamma$ is any $(K_{X^k}+B^k+M^k)$-negative extremal curve of minimal length. By the contraction theorem, $m_jP_j$ descends along these $P_j$-trivial steps, so $m_j P_j^k \cdot \Gamma \in \Z$, which in turn implies that $\inf \{P^k \cdot \Gamma >0\} > 2 \dim X/\alpha$. Therefore any step of an MMP on $K_{X^k}+B^k+M^k+\alpha P^k$ is $P^k$-trivial.
	\end{proof}
	
	\noindent \textbf{Dlt g-pairs.}
	We record further results for dlt g-pairs. They were obtained in the author's unpublished manuscript \cite{hu4} in 2021, and we restate them here with proofs.
	\begin{lem}[\text{\cite[Lemma 2.2]{hu4}}]\label{dlt-resolution}
		Let $(X,B+M)$ be a $\Q$-factorial lc g-pair with data $M'$ on $\phi: X' \to X$. Suppose there is an open subset $U \subset X$ containing the generic points of all lc centres such that $(U,B|_U)$ is log smooth and $M'|_{U'}=\phi|_{U'}^* M|_U$ where $U'$ is the inverse image. Then $(X,B+M)$ is dlt.
	\end{lem}
	\begin{proof}
		Pick a number $0< \epsilon \ll 1$ and write
		$$
		K_{X'} +B_\epsilon' +M'=\phi^*(K_X+(1-\epsilon )B +M) +F
		$$
		where $B_\epsilon',F \ge 0$ have no common components.
		
		Pick an exceptional divisor $G\ge 0$ such that $-G$ is ample$/X$ and $\|G\| \ll 1$. By construction $(X',(B_\epsilon'+G) +\alpha M')$ is a klt g-pair where $\alpha \gg 0$. Now run an MMP$/X$ on $	K_{X'} +B_\epsilon'+G +\alpha M' $. Since $M'\sim_{\R,X}M'-\phi^*M$ and $G$ contains all the exceptional divisors, by Lemma \ref{lem-P-trivial} the above MMP is $M'$-trivial for $\alpha \gg 0$. Moreover, by \cite{bchm} and Lemma \ref{lem-exc-3}, the MMP terminates with a minimal model $\pi:X'' \to X$ and contracts all exceptional prime divisors $D$ with $\mult_D (\phi^*M -M') =0$, and the data descends to $X''$.
		
		We claim that $\pi$ is an isomorphism over the generic point of every lc centre. Suppose otherwise over an lc centre $V$. Since $V$ is also an lc centre of $(X,B)$, there is a divisor $D$ over $X$ with centre $V$ and $a(D,X,B)=0$. As $X$ is $\Q$-factorial, the exceptional locus $\ex (\pi)$ has pure codimension one. Hence a prime exceptional divisor $E$ contains the centre of $D$. By construction, $\mult_E (\pi^*M -M'') >0$, so $a(E,X,B) > a(E,X,B+M)=a(E,X'',B'')$, where $K_{X''}+B''+M''=\pi^*(K_X+B+M)$. Hence $0=a(D,X,B)>a(D,X'',B'')$, which is a contradiction.
		
		The ordinary sub-pair $(X'',B'')$ is sub-dlt, since it is log smooth at the generic points of all its lc centres. There is a log resolution of this sub-pair extracting no lc places. Composing it with $\pi$, we see that $(X,B+M)$ is dlt, since the data descend to $X''$ and $\pi$ extracts no lc places.
	\end{proof}
	
	\begin{lem}[\text{\cite[Lemma 2.3]{hu4}}]\label{lem-dlt}
		Let $(X,B+M)$ be an lc g-pair with data $M'$ on $\phi: X' \to X$. Suppose there is an open subset $U \subset X$ containing the generic points of all lc centres such that $(U,B|_U)$ is log smooth and $M'|_{U'}=\phi|_{U'}^* M|_U$ where $U'$ is the inverse image. Then $X$ admits a small $\Q$-factorialisation which is isomorphic near the generic point of every lc centre.
	\end{lem}
	\begin{proof}
		There exists a log smooth dlt pair $\pi:(X'',B'') \to X$ such that $\pi$ is an isomorphism over $U$, where $K_{X''}+B''+M''=\pi^*(K_X+B+M)+F$, and $B'',F\ge 0$ have no common components, and $M''=(^*M')_{X''}$. Let $E''$ be the reduced exceptional divisor and choose $0< \epsilon \ll 1$ so that $(X'',B''+\epsilon E'')$ is dlt. Now run an MMP$/X$ on $K_{X''}+B''+ \epsilon E''$ with scaling of an ample$/X$ divisor. Since $K_{X''}+B''+ \epsilon E''=\epsilon E''+F-M''/X$ and $M''$ is the push-down of a nef and big$/X$ divisor $M'$, after replacing $M''$, we can assume that $M'' \ge 0$ and $\| M'' \| \ll1 $. By Lemma \ref{lem-exc-3}, this MMP contracts $E''$ after finitely many steps, giving a $\Q$-factorial model $X''$ with a small contraction to $X$, which is an isomorphism over $U$.
	\end{proof}
	
	\begin{prop}[\text{\cite[Lemmas 2.2 and 2.3]{hu4}}]
		Let $(X,B+M)$ be an lc g-pair with data $M'$ on $\phi: X' \to X$. Suppose there is an open subset $U \subset X$ containing the generic points of all lc centres such that $(U,B|_U)$ is log smooth and $M'|_{U'}=\phi|_{U'}^* M|_U$ where $U'$ is the inverse image. Then $(X,B+M)$ is dlt.
	\end{prop}
	\begin{proof}
		Combine Lemma \ref{dlt-resolution} and Lemma \ref{lem-dlt}.
	\end{proof}
	
	\begin{prop}[\text{\cite[Lemma 2.4]{hu4}, cf.\cite[Corollary 5.52]{kollar-mori}}]\label{lem-dlt-2}
		Let $(X,B+M)$ be a dlt g-pair. Then any lc centre is normal.
	\end{prop}
	\begin{proof}
		Following the proof of \cite[Proposition 2.43]{kollar-mori}, we can assume that $(X,B+M)$ is plt, $\Supp \{M\} \subset \Supp \{B\}$ and $B+M$ is a $\Q$-divisor. The connectedness lemma \cite[Theorem 5.48]{kollar-mori} and the proof of \cite[Proposition 5.51]{kollar-mori} then apply.
	\end{proof}
	
	\subsection{Termination with scaling versus the existence of minimal models}
	
	We begin with a lemma analogous to \cite[Theorem 4.1, Steps 2 and 4]{birkar-flip}.
	\begin{lem}[\text{cf.\cite[Thm 4.1]{birkar-flip}}]\label{mm}
		Let $(X/Z,B+M)$ be an lc g-pair and $(Y/Z,B_Y+M_Y)$ be a minimal model. Let $C \ge 0$ be an $\R$-Cartier divisor on $X$ whose support contains no lc centres and $C_Y$ be its birational transform. Then there exists a $\Q$-factorial dlt blow-up $Y' \to Y$ such that $(Y'/Z,B_{Y'} +\delta C_{Y'}+M_{Y'})$ is dlt for $0 < \delta \ll 1$, where $C_{Y'}$ is the birational transform of $C$.
		
		Moreover, if $(X/Z,B+M)$ is dlt and $C$ is ample, then there is a boundary $\Delta_{Y'}$ such that $K_{Y'}+B_{Y'}+ \delta C_{Y'}+M_{Y'} \sim_\R K_{Y'} + \Delta_{Y'}$ and $(Y',\Delta_{Y'})$ is klt.
	\end{lem}
	\begin{proof}
		Let $X\overset{p}{\leftarrow}W\overset{q}{\to}Y$ be a common log resolution of $(X,B+C+M)$ and $(Y,B_Y+C_Y+M_Y)$. Let $B_W=B^\sim+\Exc(p)_{\mathrm{red}}$, where $B^\sim$ denotes the birational transform, and set
		$$
		\begin{aligned}
			E&=p^*(K_X+B+M)-q^*(K_Y+B_Y+M_Y),\\
			E'&=K_W+B_W+M_W-p^*(K_X+B+M).
		\end{aligned}
		$$
		
		Run an MMP$/Y$ on $(W,B_W+M_W)$. Since
		$$
		K_W+B_W+M_W=q^*(K_Y+B_Y+M_Y)+E+E',
		$$
		where $E+E'\geq0$ is exceptional$/Y$, Lemma \ref{lem-exc-3} contracts $\Supp(E+E')$ after finitely many steps. So $Y'\to Y$ is a dlt blow-up, and $C_{Y'}$ contains no lc centres since $C_W$ contains none. Note that, for $0<\delta\ll1$, this finite MMP is also an MMP on $(W,B_W+\delta C_W+M_W)$, where $C_W$ is the birational transform, so $(Y',B_{Y'}+\delta C_{Y'}+M_{Y'})$ is dlt.
		
		For the moreover part, let $0<\epsilon\ll1$ and then $0<\delta\ll\epsilon$. By an ample perturbation, there is a klt boundary $\Delta$ with $K_X+\Delta\sim_\R K_X+B+\delta C+M$, sufficiently close on $W$ so that
		$$
		K_W+\Delta_W=p^*(K_X+\Delta)+(1-\epsilon)E'
		$$
		defines an effective klt boundary. The above finite MMP is also a $(K_W+\Delta_W)$-MMP, so its push-down $\Delta_{Y'}$ is klt. Since $E'$ is contracted, we obtain
		$$
		K_{Y'}+\Delta_{Y'}\sim_\R K_{Y'}+B_{Y'}+\delta C_{Y'}+M_{Y'}.
		$$
		
	\end{proof}
	
	Now we give a special case of \cite[Theorem 4.1]{birkar-flip} in the setting of g-pairs.
	\begin{lem}\label{lem-sp-term}
		Let $(X/Z,B+M)$ be a dlt g-pair which has a minimal model. Given an ample divisor $A$ and a decreasing sequence of numbers $\lambda_1 \ge \lambda_2 \ge \ldots \ge \lambda_i \ge \ldots$ such that $\lim \lambda_i =0$, suppose there is a sequence of birational maps among g-pairs
		$$
		(X,B+M) \bir \cdots \bir (X^i,B^i+M^i) \bir (X^{i+1},B^{i+1}+M^{i+1}) \bir \cdots
		$$
		satisfying:
		\begin{enumerate}
			\item $(X^i,(B^i+\lambda_i A^i)+M^i)$ is a weak lc model of $(X,(B+\lambda_i A)+M)$ on which the birational transform $A^i$ of $A$ is $\R$-Cartier;
			
			\item $K_{X^i} + B^i+\lambda A^i+M^i$ is nef$/Z$ if $\lambda_{i}\ge \lambda \ge \lambda_{i+1}$.
		\end{enumerate}
		Then $K_{X^i} + B^i+M^i$ is nef$/Z$ for all $i \gg 0$.
	\end{lem}
	\begin{proof}
		Let $(Y/Z,B_Y+M_Y)$ be a minimal model and $L_Y=K_Y+B_Y+M_Y$. By Lemma \ref{mm}, after a dlt blow-up, $(Y,B_Y+\delta A_Y+M_Y)$ is dlt and its adjoint is $\R$-linearly equivalent to that of a klt pair with big boundary for $0<\delta\ll1$. By a small perturbation within its $\R$-linear equivalence class, we can assume that this big klt boundary contains $\Supp\{L_Y\}_{ir}$. By Lemma \ref{lem-P-trivial}, there is an $L_Y$-trivial MMP on $(1+\alpha)L_Y+\delta A_Y$ for $\alpha\gg0$, which terminates by \cite{bchm}. Replacing $Y$ and rescaling $\delta$ by $1/(1+\alpha)$, by convexity we get a weak lc model of $(X,B+tA+M)$ for every $0\leq t\leq\delta$.
		
		Let $L_i(t)=K_{X^i}+B^i+tA^i+M^i$ and $L_Y(t)=L_Y+tA_Y$. For $i\gg0$ with $0<\lambda_i<\delta$, both $(X^i,B^i+\lambda_iA^i+M^i)$ and $(Y,B_Y+\lambda_iA_Y+M_Y)$ are weak lc models of $(X,B+\lambda_iA+M)$, so we have
		$$
		{}^*L_i(\lambda_i)={}^*L_Y(\lambda_i).
		$$
		Moreover, $({}^*L_Y(t))_{X^i}=L_i(t)$ for $0\leq t\leq\delta$: no prime divisor of $X$ contracted by $Y$ occurs on $X^i$, and each prime divisor of $X^i$ exceptional over $X$ is an lc place of $Y$ at $\lambda_i$, hence throughout the interval by linearity of discrepancies. Since $L_Y(t)$ is nef, by the negativity lemma we deduce
		$$
		{}^*L_i(t)-{}^*L_Y(t)\geq0\qquad(0\leq t\leq\delta).
		$$
		This difference is affine in $t$ and vanishes at the interior point $\lambda_i$, so it vanishes throughout the interval. For $t=0$, we see that $L_i(0)$ is nef. If $\lambda_i=0$, apply (1).
	\end{proof}
	
	As an immediate corollary we obtain the following.
	\begin{cor}[\text{cf. \cite[Theorem 1.9]{birkar-flip}}]\label{cor-term}
		Let $(X/Z,B+M)$ be a $\Q$-factorial dlt g-pair. Suppose $(X/Z,B+M)$ has a minimal model. Then any MMP on $K_X+B+M$ with scaling of an ample divisor terminates with a minimal model.
	\end{cor}
	\begin{proof}
		By \cite{bchm} we see $\lambda =0$. Apply Lemma \ref{lem-sp-term}.
	\end{proof}
	
	We generalise weak lc models to weak minimal models, extending the definition for usual pairs in the unpublished manuscript \cite{hu}.
	\begin{defn}[{Weak minimal models, cf. \cite[Lemma 2.9]{hu}}]\label{defn-wmm}
		A projective lc g-pair $(W,B_W+M_W)$ with data $M_{W'}$ is a \emph{weak minimal model} of a projective lc g-pair $(X,B+M)$ if:
		\begin{itemize}
			\item $W$ is birational to $X$ and $^*M'=^*M_{W'}$ as b-divisors;
			
			\item $K_W+B_W+M_W$ is nef;
			
			\item $a(D,X,B+M) \le a(D,W,B_W+M_W)$ for every prime divisor $D$ on $X$; and
			
			\item $a(D,X,B+M) +\sigma_D(K_X+B+M) \ge a(D,W,B_W+M_W)$ for every prime divisor $D$ on $W$.
		\end{itemize}
	\end{defn}
	
	Now we prove the equivalence between the existence of minimal models and weak minimal models for NQC g-pairs (cf. \cite[Lemma 2.9]{hu}). However, it is not known for arbitrary g-pairs.
	
	\begin{prop}\label{prop-wmm}
		Let $(X,B+M)$ be a projective lc g-pair. Let $(W,B_W+M_W)$ be a weak minimal model of $(X,B+M)$. Then
		$$
		a(D,X,B+M) +\sigma_D(K_X+B+M) = a(D,W,B_W+M_W)
		$$
		for every prime divisor $D$ over $X$. In particular, if $K_W+B_W+M_W$ is NQC (for example, semi-ample), then $(X,B+M)$ has a minimal model.
	\end{prop}
	\begin{proof}
		After replacing $(X,B+M)$ with a log smooth model, we can assume that $g:X \bir W$ is a morphism. Since $a(D,X,B+M) +\sigma_D(K_X+B+M) \ge a(D,W,B_W+M_W)$ for every prime divisor $D$ on $W$, choose $0 \le B' \le B$ such that $\mu_D B'= \mu_D B_W$ if $D $ is a divisor on $ W$, and $\mu_D B'= \mu_D B$ otherwise. Here we identify prime divisors of $W$ with their birational transforms on $X$. We have $B- B' \le N_\sigma(K_X+B+M)$, so a minimal model of $(X,B+M)$ is a weak lc model of $(X,B'+M)$. The divisor $N:= K_X+B'+M -g^*(K_W+B_W+M_W) $ is effective and exceptional$/W$, hence $N=N_\sigma(K_X+B'+M)$. Therefore	$$
		a(D,X,B+M) +\sigma_D(K_X+B+M) = a(D,W,B_W+M_W)
		$$
		for every prime divisor $D$ over $X$.
		
		In particular, if $K_W+B_W+M_W$ is NQC, then we can run an MMP on $K_X+B+M+\alpha g^*(K_W+B_W+M_W)$ with scaling of an ample divisor for some $\alpha \gg 0$. This MMP is $g^*(K_W+B_W+M_W)$-trivial and contracts $\Supp N_\sigma(K_X+B+M)$ after finitely many steps, hence terminates with a minimal model.
	\end{proof}
	
	\subsection{MMP on \texorpdfstring{$K_X+B+M+\alpha P$}{KX+B+M+aP} for \texorpdfstring{$\alpha \gg 0$}{a >> 0}}
	
	We begin with a technical result. Given a divisor $D$, we denote its irrational fractional part (the  fractional part with irrational coefficients) by $\{D\}_{ir}$. Any divisor $D$ can be written as $D=\rddown{D}+\{D\}_{ir}+\{D\}_{r}$ where $\{D\}_{r}$ denotes the rational fractional part.
	\begin{thm}\label{thm-klt}
		Let $(X/Z,B+M)$ be a $\Q$-factorial dlt g-pair with data $M'$ and let $P$ be a nef$/Z$ divisor. Write $N=K_X+B+M-P$. For a divisor $E$, denote its coefficientwise positive part by $E^+$. Suppose that one of the following holds:
		\begin{enumerate}
			\item $P\equiv\pm(tK_X-C-sN+D)$, where $t\in\Q_{\geq 0}$, $s\geq 0$, $sN$ is a $\Q$-Cartier $\Q$-divisor, $C,D\geq 0$, $\Supp C\subseteq\Supp B$, and $\Supp\bigl((D-C-tB)^+\bigr)$ contains no lc centres of $(X,B+M)$; or
			\item there is a divisor $Q\equiv P$ such that $\Supp\{Q\}_{ir}\subseteq\Supp\{B\}$.
		\end{enumerate}
		Then any MMP on $K_X+B+M+\alpha P$ is $P$-trivial for $\alpha\gg 0$.
	\end{thm}
	\begin{proof}
		Let $L=K_X+B+M=P+N$.
		
		\emph{Case (1).} Pick $\delta\in\{1,-1\}$ and let
		$$
		G=tK_X-C+D,\qquad Q=\delta(G-sN)\equiv P,\qquad E=D-C-tB.
		$$
		For a sufficiently small rational number $\eta>0$, let
		$$
		B_\eta=\frac{B-\eta C+\eta D}{1+\eta t}
		=B+\frac{\eta}{1+\eta t}E,\qquad M_\eta'=\frac{M'}{1+\eta t}.
		$$
		Since $\Supp E^+$ contains no lc centres and the negative part of $E$ is supported on $\Supp B$, we see that $(X,B_\eta+M_\eta)$ is a dlt g-pair, as it is obtained by adding a small multiple of $E^+$, decreasing the boundary and scaling down the nef data.
		
		As $tK_X$ and $sN$ are $\Q$-divisors, $\Supp\{Q\}_{ir}\subseteq\Supp B_\eta$. Moreover, $\lfloor B_\eta\rfloor\leq\lfloor B\rfloor$ and $\mult_S(D-C)=t$ for each component $S$ of $\lfloor B_\eta\rfloor$, so $\mult_SQ\in\Q$, which in turn implies that $\Supp\{Q\}_{ir}\subseteq\Supp\{B_\eta\}$.
		
		Let $L_\eta=K_X+B_\eta+M_\eta$. Since $G\equiv sL+(\delta-s)P$,
		$$
		L_\eta=\frac{L+\eta G}{1+\eta t}\equiv aL+bP,\qquad
		a=\frac{1+\eta s}{1+\eta t}>0,\quad b=\frac{\eta(\delta-s)}{1+\eta t}.
		$$
		Since $L_\eta+(a\alpha-b)Q\equiv a(L+\alpha P)$, Lemma \ref{lem-P-trivial} makes these MMPs $Q$-trivial, hence $P$-trivial, for $\alpha\gg0$.
		
		\emph{Case (2).} Because MMP is a numerical process, we can apply Lemma \ref{lem-P-trivial} to $(X/Z,B+M)$ and $Q\equiv P$.
	\end{proof}
	
	\begin{rem}\label{rem-nqc-extension}
		(1) If the nef data has an NQC summand, then this summand does not affect our arguments, because we can approximate the coefficients of its NQC decomposition by rational numbers. So we only need to look at the remaining part. We explain this for Theorem \ref{thm-klt}. Suppose that $M'=M_1'+M_2'$, where $M_1'$ is NQC and $M_2'$ is nef, and let $M_1,M_2$ be their push-downs. We still define $N=K_X+B+M-P$ with $M=M_1+M_2$. Then Theorem \ref{thm-klt} still holds if we replace the numerical condition in (1) by (i) and the one in (2) by (ii):
		\begin{enumerate}
			\item[(i)] $P\equiv\pm(tK_X-C-sN+D)-rM_1$ for some $r\in\Q$; or
			\item[(ii)] $Q\equiv P-rM_1$ for some $r\in\Q$, and $\Supp\{Q\}_{ir}\subseteq\Supp\{B\}$.
		\end{enumerate}
		Note that we keep all other hypotheses, including the support condition in (1).

		For (ii), we fix an NQC decomposition of $M_1'$ and approximate its coefficients together with the boundary so that $N$ does not change, while $M_2'$ is fixed. Since each nef summand descends near the generic points of lc centres, the nef data remain nef and the g-pair remains dlt after small perturbations. So we can argue as in Lemma \ref{lem-P-trivial}. For (i), we follow the proof of Theorem \ref{thm-klt} and replace the nef data by
		$$
		M_\eta'=\frac{M'-\eta\delta rM_1'}{1+\eta t},
		\qquad \delta\in\{1,-1\}.
		$$
		We then apply (ii) and the same numerical comparison. The same philosophy applies to all results of this paper, and we will not repeat it later.

		(2) In the notation of Theorem \ref{thm-klt}, if
		$$
		\Supp\{M-N\}_{ir}\subseteq\Supp\{B\},
		$$
		then condition (2) holds, because we can take $Q=P=K_X+B+M-N$. In particular, it is enough to require
		$\Supp\{M\}_{ir}\cup\Supp\{N\}_{ir}\subseteq\Supp\{B\}$.

		(3) If $P$ is NQC, then any MMP on $K_X+B+M+\alpha P$ is $P$-trivial for $\alpha\gg0$ by \cite[Lemma 3.22]{hanli}.
	\end{rem}
	
	We deduce the following two corollaries from Theorem \ref{thm-klt}(1).
	\begin{cor}\label{cor-klt}
		Let $(X/Z,B+M)$ be a $\Q$-factorial dlt g-pair and $P$ be a nef$/Z$ divisor. Write $N=K_X+B+M-P$. Suppose one of the following holds:
		\begin{enumerate}
			\item[(1a)] $P \equiv D \ge 0$ and $\Supp D$ contains no lc centres; or
			
			\item[(1b)] $M \equiv D -sN$, where $D \ge 0$ and $\Supp D$ contains no lc centres, $N$ is a $\Q$-Cartier $\Q$-divisor and $s \in \Q_{\ge -1}$.
			
		\end{enumerate}
		Then any MMP on $K_X+B+M + \alpha P$ is $P$-trivial for $\alpha \gg 0$.
	\end{cor}
	\begin{proof}
		For (1a), we apply Theorem \ref{thm-klt}(1) with $t=s=0$ and $C=0$.
		For (1b), we have
		$$
		P\equiv K_X+B+D-(1+s)N.
		$$
		Apply Theorem \ref{thm-klt}(1) with effective divisor $B+D$ and coefficient $1+s\geq0$ of $N$. Note that its support condition reduces to that on $D$.
	\end{proof}
	
	In particular, the next corollary follows when $P$ is the positive part of the Nakayama--Zariski decomposition.
	\begin{cor}\label{cor-klt'}
		Let $(X,B+M)$ be a projective klt g-pair. Assume $K_X+B+M$ birationally has a Nakayama--Zariski decomposition with nef positive part. Suppose either $M \equiv D \ge 0$ or $K_X+B+M \equiv D \ge 0$. Then $(X,B+M)$ has a minimal model.
	\end{cor}
	\begin{proof}
		
		Passing to a log resolution and making a sufficiently small boundary perturbation supported on the negative part and the exceptional divisors, we can assume that $X$ is smooth and $K_X+B+M=P+N$ is its Nakayama--Zariski decomposition, with $P$ nef and $N\geq0$ a $\Q$-divisor (cf.\ the proof of \cite[Theorem 1.1]{bhzariski}). These replacements leave $P$ unchanged and every exceptional divisor is contained in $\Supp N$.
		
		If originally $M\equiv D\geq0$, we still have $M\equiv D'-sN$ for some $D'\geq0$ and $s\in\Q_{\geq0}$. In the other case, $P$ remains numerically equivalent to an effective divisor. By Corollary \ref{cor-klt}, run a $P$-trivial MMP on $K_X+B+M+\alpha P$ for $\alpha\gg0$. It contracts $N$ after finitely many steps and terminates with a model $(Y,B_Y+M_Y)$ on which $K_Y+B_Y+M_Y=P_Y$ is nef. Since all boundary changes are supported on $N$, this is also a minimal model of the original pair.
	\end{proof}
	We study the lc case in Section \ref{sec-P-mmp}.\\
	
	\section{\texorpdfstring{$P$-trivial MMP}{P-trivial MMP} and its applications}\label{sec-P-mmp}
	
	\subsection{\texorpdfstring{$P$-trivial MMP}{P-trivial MMP}}
	The following lemma allows us to run an MMP with scaling of a prescribed divisor, such as an ample divisor, for arbitrary g-pairs.
	\begin{lem}\label{lem-P-mmp}
		Let $(X/Z,B+M)$ be a $\Q$-factorial dlt g-pair. Let $C $ be a divisor such that $(X,(B+C)+M)$ is dlt and $K_X+(B+C)+M $ is nef. Suppose that $\Supp \{M \}_{ir} \subseteq \Supp \{B+ C \}$. Let
		$$
		\lambda=\inf \{t\ge 0 | K_X+B+tC +M \text{ is nef}\}.
		$$
		If $\lambda>0$, then there exists a $(K_X+B+M)$-negative extremal curve $\Gamma$ such that $(K_X+B+\lambda C+M )\cdot \Gamma =0$. Moreover, we can run an MMP$/Z$ on $K_X+B+M$ with scaling of $C$.
	\end{lem}
	\begin{proof}
		Replacing $(B,C)$ by $(B+\epsilon C,(1-\epsilon)C)$ for a general $0<\epsilon\ll1$, we can assume that $\Supp\{M\}_{ir}\subseteq\Supp\{B+C\}\subseteq\Supp\{B\}$. Let $\Gamma_i$ be extremal curves of minimal length generating the $(K_X+B+M)$-negative rays, and define $\lambda_i>0$ by $(K_X+B+\lambda_iC+M)\cdot\Gamma_i=0$. Then $\lambda=\sup_i\lambda_i$ and
		$$
		\frac1{\lambda_i}
		=\frac{(K_X+B+C+M)\cdot\Gamma_i}{-(K_X+B+M)\cdot\Gamma_i}+1.
		$$
		Because $\Supp\{M\}_{ir}\subseteq\Supp\{B+C\}\subseteq\Supp\{B\}$ and both $(X,B+M)$ and $(X,B+C+M)$ are dlt, there are convex decompositions $B=\sum_j\alpha_jB_j$ and $B+C=\sum_k\beta_k\Delta_k$ with $\Q$-linearly independent coefficients such that $B_j+M,\Delta_k+M$ are $\Q$-divisors and $(X,B_j+M)$, $(X,\Delta_k+M)$ are dlt. By the length of extremal rays, there is an integer $m>0$ such that
		$$
		(K_X+B_j+M)\cdot\Gamma_i=\frac{n_{i,j}}m,\qquad
		(K_X+\Delta_k+M)\cdot\Gamma_i=\frac{n'_{i,k}}m,
		$$
		where $n_{i,j},n'_{i,k}\in\Z_{\geq-2m\dim X}$. Hence
		$$
		\frac1{\lambda_i}=\frac{\sum_k\beta_kn'_{i,k}}{-\sum_j\alpha_jn_{i,j}}+1,
		\qquad 0<-\sum_j\alpha_jn_{i,j}\leq2m\dim X,
		$$
		Therefore $\lambda=\lambda_i$ for some $i$.

		Now we prove the moreover part. Given an MMP$/Z$ on $K_X+B+M$ with scaling of $C$, say $X\bir X^2\bir\cdots\bir X^n$, we show that the above conclusion holds for $X^n$. Note that $(X^n,(B^n+\lambda_{n-1}C^n)+M^n)$ is lc but not necessarily dlt. If $\lambda_n<\lambda_{n-1}$, then the above argument works after replacing $C$ by $tC^n$ for some $\lambda_n<t<\lambda_{n-1}$. Otherwise, let $m$ be the first index with $\lambda_m=\lambda_n$, and we take the above decomposition of $B^m+\lambda_mC^m$. Since the intervening steps are trivial for its rational adjoints, their transforms on $X^n$ are still lc, and we can apply the above argument.
	\end{proof}
	
	\begin{lem}\label{lem-P-mmp-2}
		With the notation and assumptions of Lemma \ref{lem-P-mmp}, let $P$ be a nef divisor with $\Supp \{P \}_{ir} \subseteq \Supp \{B\} $ and
		$$
		\lambda(\alpha)=\inf \{t\ge 0 | K_X+B+tC +M+\alpha P \text{ is nef}\}.
		$$
		Then there exists a number $l>0$ such that $\lambda(\alpha)$ is independent of $\alpha$ for $\alpha\geq l$. Moreover, any MMP$/Z$ on $K_X+B+M+lP$ with scaling of $C$ is $P$-trivial and is also an MMP$/Z$ on $K_X+B+M+\alpha P$ with scaling of $C$ for every $\alpha\geq l$.
	\end{lem}
	\begin{proof}
		By Lemma \ref{lem-P-trivial} there exists a number $l>0$ such that $P \cdot \Gamma =0$ for any $(K_X+B+M+ lP)$-negative extremal curve $\Gamma$ of minimal length. For $\alpha\geq l$, since $P$ is nef, we have $\lambda(\alpha)\leq\lambda(l)$. Suppose that the inequality is strict, so that there is an extremal curve $\Gamma$ on which $K_X+B+\lambda(\alpha)C+M+lP$ is negative. Since $K_X+B+C+M+lP$ is nef and $0\leq\lambda(\alpha)<\lambda(l)\leq1$, the ray is also $(K_X+B+M+lP)$-negative. So $P\cdot\Gamma=0$, which contradicts the nefness of $K_X+B+\lambda(\alpha)C+M+\alpha P$. Since $l$ is preserved under MMP, the rest is clear.
	\end{proof}
	
	\begin{defn}[$P$-nefness]\label{defn-P-nef}
		Let $X$ be a normal variety projective over $Z$ and $P$ be a nef divisor. An $\R$-Cartier divisor $D$ is said to be \emph{$P$-nef (resp. $P$-pseudo-effective)} if for some ample$/Z$ divisor $A$, for any sufficiently small number $\epsilon >0$, there exists a number $\alpha>0$ such that $D+\epsilon A +\alpha P$ is nef (resp. pseudo-effective).
	\end{defn}
	
	Note that the above definition is independent of the choice of the ample divisor $A$. An $\R$-Cartier divisor $D$ is $P$-nef if and only if
	$$
	D\cdot\gamma\geq0\quad\text{for every }\gamma\in\NEbar(X/Z)\cap P^\perp.
	$$
	Indeed, on the compact slice $A\cdot\gamma=1$, the divisor $D+\epsilon A$ is positive near $P^\perp$, while $P\cdot\gamma$ is bounded away from zero on the complement. The converse follows by letting $\epsilon$ tend to zero.
	
	\begin{exa}[$P$-nefness and curves]\label{exa-P-nef-curves}
		Let $X$ and $P=M$ be as in Example \ref{exa: nonvanishing fail}. Then $P^2=0$ and the ray $\R_{\geq0}[P]$ is not rational. Every curve on the abelian surface $X$ is nef, so by the Hodge index theorem we have $P\cdot C>0$ for every curve $C$. Let $H$ be ample and $D=-H$. There are no $P$-trivial curves, but
		$$
		(D+\epsilon H+\alpha P)\cdot P=(\epsilon-1)H\cdot P<0
		$$
		for $0<\epsilon<1$ and every $\alpha>0$. So $D$ is not $P$-nef. In particular, one cannot replace classes in the closed cone by actual curves in the criterion above.
	\end{exa}
	
	The following lemma shows that $P$-nefness and $P$-pseudo-effectiveness are preserved under pull-backs.
	\begin{lem}\label{lem-pullbacks}
		Let $X$ be a normal variety projective over $Z$, $P$ be a nef divisor and $f: Y \to X$ be a projective (resp. surjective projective) morphism. If an $\R$-Cartier divisor $D$ is $P$-nef (resp. $P$-pseudo-effective), then $f^*D$ is $f^*P$-nef (resp. $f^*P$-pseudo-effective).
	\end{lem}
	\begin{proof}
		Given any $\epsilon > 0$, note that $f^*D +\epsilon A_Y +\alpha f^*P=\epsilon (A_Y -\delta f^*A)+f^*(D+ \epsilon \delta A +\alpha P)$ is nef (resp. pseudo-effective) for $\delta \ll 1$ and $\alpha \gg 0$.
	\end{proof}
	
	Because ampleness is not necessarily preserved under an MMP, we introduce a notion that allows us to reindex the MMP.
	\begin{defn}[Birationally ample divisors]
		Let $(X/Z,B+M)$ be an lc g-pair. A big divisor $A \ge 0$ is said to be \emph{birationally ample with respect to $(X,B+M)$} if
		\begin{itemize}
			\item $(X/Z,(B+A)+M)$ is lc; and
			
			\item The augmented base locus $\mathbf{B}_{+}(A/Z)$ contains no lc centres of $(X,(B+A)+M)$.
		\end{itemize}
	\end{defn}
	
	\begin{defn}[$P$-trivial MMP]\label{defn-P-mmp}
		Let $(X/Z,B+M)$ be a $\Q$-factorial dlt g-pair and $P$ be a nef divisor. Let $A$ be a birationally ample divisor such that the pair $(X,(B+A)+M)$ is dlt and $\Supp \{P\} \cup \Supp \{M \} \subseteq \Supp \{B+A\} $. If $K_X+B+M$ is not $P$-nef, then for every $0< \lambda \ll 1$, there is a number $\alpha \gg 0$ depending on $\lambda$ such that we can run an MMP with scaling of $A$ on $K_X+(B+\lambda A +M)+\alpha P$ which is $P$-trivial and terminates with a good minimal model $(X^j,(B^j+ \lambda A^j)+M^j + \alpha P^j)$ or a Mori fibre space by \cite{bchm}.
		
		If the above MMP reaches a model $(X^i,B^i+M^i)$ on which $K_{X^i}+B^i+M^i$ is $P^i$-nef, or we reach a Mori fibre space, then we stop. We call the previous model a \emph{$P$-minimal model}. Otherwise we pick a sufficiently small number $\lambda_2 < \lambda_1= \lambda$ and continue the process of an MMP on $K_{X^i}+B^i+\lambda_2 A^i + M^i+\alpha_2 P^i$ which also terminates with a good minimal model.
		
		We continue the process by picking a decreasing sequence $\lambda_1 > \ldots > \lambda_j > \ldots $ with $\lim \lambda_j =0$. Since an MMP on $K_X+B+\lambda_j A+M + \alpha_j P$ with scaling of $A$ is also an MMP on $K_X+B+\lambda_{j+1} A+M + \alpha_{j+1} P$ with scaling of $A$ by Lemmas \ref{lem-P-mmp} and \ref{lem-P-mmp-2}, we obtain an MMP with scaling
		$$
		X  \bir  \ldots \bir X^i \bir X^{i+1} \bir\ldots.
		$$
		We call such an MMP a \emph{$P$-trivial MMP (with scaling of $A$)}.
	\end{defn}

	The assumptions on $A$ in Definition \ref{defn-P-mmp} are only used to run the MMP. More generally, we also call a sequence of $P$-trivial divisorial contractions and flips $X^i\bir X^{i+1}$ between $\Q$-factorial dlt g-pairs a \emph{$P$-trivial MMP with scaling of $A$} if there are numbers $\lambda_1\geq\lambda_2\geq\cdots$, with $\lim\lambda_i=0$ when the sequence is infinite, such that for each $i$ there is a number $\beta\geq0$ satisfying the following: the map $X^i\bir X^{i+1}$ is a step of an MMP on $K_{X^i}+B^i+M^i+\beta P^i$ with scaling of $A^i$, and $K_{X^i}+B^i+\lambda A^i+M^i+\beta P^i$ is nef for $\lambda_{i+1}\leq\lambda\leq\lambda_i$. For example, the maps induced on lc centres in Theorem \ref{thm-sp-term} form such a sequence.
	
	\begin{exa}[A $P$-minimal model with no nef threshold]\label{exa-P-minimal-no-threshold}
		Let $T$ be the rational surface obtained by blowing up ten points of a cuspidal plane cubic in \cite[Appendix II, Theorem 11.1]{cantatdupont}. It admits an automorphism $f$ of positive entropy whose only irreducible periodic curve is $C\sim-K_T$. There is a nef divisor $Q$ with $f^*Q\equiv\rho Q$, where $\rho>1$. By \cite[Proposition 6.1]{cantatdupont}, $C$ is the only $Q$-trivial curve. Moreover,
		$$
		K_T^2=-1,\qquad Q^2=K_T\cdot Q=0,\qquad K_T\cdot C=1.
		$$
		Let $H$ be an ample divisor. For every $\epsilon>0$, by the cone theorem there are only finitely many $(K_T+\epsilon H)$-negative extremal rays, and all of them are $Q$-positive. Hence $K_T+\epsilon H+a_\epsilon Q$ is nef for some $a_\epsilon>0$.
		
		Consider the klt g-pair $(T,0+Q)$ with data on $T$, and let $P=Q$. Then $K_T+Q$ is $P$-nef, so $T$ is a $P$-minimal model. However,
		$$
		(K_T+Q+\alpha P)^2=-1
		$$
		for every $\alpha\geq0$. Hence this model is not a minimal model of $(T,0+(1+\alpha)Q)$ for any $\alpha\geq0$.
		
		This can also occur at the end of a non-empty $P$-trivial MMP. Let $\mu:\widetilde T\to T$ be the blow-up of a point of $T$, and let $E$ be the exceptional curve. Let $\widetilde P=\widetilde M=\mu^*Q$. Then we have
		$$
		K_{\widetilde T}+\widetilde M=\mu^*(K_T+Q)+E.
		$$
		So the contraction of $E$ is $\widetilde P$-trivial and $(K_{\widetilde T}+\widetilde M)$-negative, and we arrive at the above $P$-minimal model.

		Finally, let $F$ be a $(-1)$-curve on $T$ and $F_m=f^{-m}(F)$. One calculates that
		$$
		K_T\cdot F_m=-1,\qquad Q\cdot F_m=\rho^{-m}Q\cdot F\longrightarrow0.
		$$
		So the positive intersections with $K_T$-negative extremal curves can tend to zero. We also note that the ray of $Q$ is irrational and $Q^2=0$. Hence by the Hodge index theorem, $Q$ is not NQC.
	\end{exa}

	From now on, we always run $P$-trivial MMPs with scaling of birationally ample divisors, unless we say otherwise.
	
	\begin{rem}\label{rem-P-mori}
		Let $(X/Z,B+M)$ be a $\Q$-factorial dlt g-pair, let $P$ be nef$/Z$, and let $L=K_X+B+M$. A $P$-trivial MMP on $L$ ends with a Mori fibre space if and only if $L$ is not $P$-pseudo-effective. Indeed, in the latter case, $L+\epsilon A+\alpha P$ is not pseudo-effective for every $\alpha\geq0$ and all sufficiently small $\epsilon>0$, where $A$ is the scaling divisor. Therefore, for sufficiently small $\lambda_j$, the corresponding MMP in Definition \ref{defn-P-mmp} ends with a Mori fibre space by \cite{bchm}. The converse follows from the $P$-triviality and $L$-negativity of the final contraction.
	\end{rem}
	
	It remains to ask when a $P$-trivial MMP terminates in the $P$-pseudo-effective case, and when $P$-pseudo-effectiveness of $K_X+B+M$ implies that $K_X+B+M+\alpha P$ is pseudo-effective for some $\alpha\geq0$.
	
	We give a $P$-trivial version of Corollary \ref{cor-term}.
	
	\begin{prop}[\text{cf.Corollary \ref{cor-term}}]
		Let $(X/Z,B+M)$ be a $\Q$-factorial dlt g-pair and $P$ be a nef divisor. Suppose $(X/Z,B+(M+cP))$ has a minimal model $(Y,B_Y+(M_Y+cP_Y))$ for some $c \ge 0$ and the birational map $X \bir Y$ is $P$-trivial. Then any $P$-trivial MMP on $K_X+B+M$ terminates with a minimal model of $(X,B+M+ \alpha P)$ for every $\alpha \ge c$.
	\end{prop}
	
	The proof is analogous to that of Lemma \ref{lem-sp-term}, using $P$-triviality to identify the pullbacks of $P$, and is left to the reader.
	
	\subsection{A special termination on \texorpdfstring{$P$-trivial MMP}{P-trivial MMP}}
	
	\begin{defn}[Local Cartier index]
		Let $X$ be a normal variety of klt type of dimension $d$, that is, there is a boundary $B$ such that $(X,B)$ is klt. Let $x\in X$ be a point of codimension two and put $V=\overline{\{x\}}$. At a general closed point $x_0\in V$, take a general surface section $T$ transverse to $V$. Then $T$ is normal and $\Q$-factorial near $x_0$. We define the \emph{local Cartier index at $x$} as the minimal positive integer $l$ such that $lD$ is Cartier near $x_0$ for every integral Weil divisor $D$ on $T$ near $x_0$.
		
		If $V=\overline{\{x\}}\subset S$ for a prime divisor $S$ and $(X,S)$ is plt near the generic point of $V$, then $l$ equals the Cartier index of $K_X+S$ there and is independent of the chosen surface section (see \cite[proof of Proposition 4.9]{birkarzhang} and \cite[Proposition 3.9]{sho}).
	\end{defn}
	
	We immediately obtain the following.
	\begin{lem}\label{lem-index-co}
		Let $X$ be a normal variety of klt type, $S$ be a hypersurface such that $(X,S)$ is plt, and $V$ be a prime divisor of $S$. Let $D=\sum_{i} a_i D_i$ be an $\R$-Cartier divisor with $S \nsubseteq \Supp D$ and $l$ be the local Cartier index at (the generic point of) $V$. Then the multiplicity $\mu_V(D|_S)$ is of the form $\sum_{i} a_i \frac{m_i}{l}$, where $m_i$ are non-negative integers with $m_i>0$ if and only if $D_i$ passes through $V$.
	\end{lem}
	\begin{proof}
		Cutting by general hyperplanes, we can assume $\dim X = 2$. Let $\pi: X' \to X$ be a log resolution of $(X,S+D)$. We have $\mu_{V'}(\pi^*D|_{S'})= \mu_{E} \pi^*D$ for some divisor $E$, where $V'$ is a component of $S' \bigcap E$ dominating $V$. Since $S$ is normal, $V'$ is unique and $V' \to V$ is birational, and hence $\mu_V(D|_S)=\mu_{V'}(\pi^*D|_{S'})$.
	\end{proof}
	
	The following follows from the proof of \cite[Remark 4.8 and Proposition 4.9]{birkarzhang}.
	\begin{lem}\label{lem-coeff}
		Let $(X,B+M)$ be a dlt g-pair with data $M'$ on $\pi: X' \to X$ and $S \subset \rddown{B}$ be a component. Then there is an adjunction formula
		$
		K_S+B_S+M_S=(K_X+B+M)|_S$ and for any prime divisor $V$ of $S$ we have $$\mu_V(B_S) = \frac{l-1}{l} + \sum_{k} b_k \frac{d_k}{l} + \mu_V\bigl((\pi|_{S'})_*((\pi^*M-M')|_{S'})\bigr)
		$$
		where $b_k$ are the coefficients of the components $B_k\neq S$ of $B$ passing through $V$, $d_k\in\Z_{>0}$, $l$ is the local Cartier index at $V$ and $S'$ is the birational transform of $S$. Note that $\pi^*M-M'$ is well-defined near the generic point of $V$ since $X$ is $\Q$-factorial in codimension two.
	\end{lem}
	
	Recall the following easy fact (cf. \cite[Lemma 2.45]{kollar-mori}).
	\begin{lem}
		Let $(X,B)$ be a log smooth sub-lc pair and $E$ be a prime divisor over $X$ such that $a(E,X,B) <1$. Then the centre $c_X(E)$ is a stratum of $(X, \Supp B^{>0})$.
	\end{lem}
	
	\begin{lem}[A canonical blow-up]\label{lem-canonical-blow-up}
		Let $(X,B+M)$ be a dlt g-pair. Let $\pi:(W,B_W+M_W) \to X$ be a log resolution of $(X,B+M)$
		such that the components of $\{B_W^{>0}\}$ are disjoint, where $K_W+B_W+M_W=\pi^*(K_X+B+M)$. Let $E$ be a prime divisor over $W$ such that $0 < a(E,X,B+M) < 1$. Then $E$ is obtained by a sequence of blow-ups
		$$
		E \subset W_n \to \ldots \to W_0=W
		$$
		along a stratum $V_i \subset W_i$ of $\Supp B_{W_i}^{>0}$ and $V_0$ is a component of $D \bigcap (\bigcap_i S_i)$ for some $D \subset \{B_W^{>0}\}$ and some $S_i \subset B_W^{=1}$. Moreover, if we write $b:=\mu_D B_W$ and $b_E:=1 - a(E,X,B+M)$, then $b_E= db-d+1 \ge 0$ for some $d \in \Z_{>0}$.
		
	\end{lem}
	\begin{proof}
		Note that if $a(E,X,B+M) <1$, then $c_{W_i}(E)$ is a stratum of $B_{W_i}^{>0}$ by the previous lemma and $E \subset W_n$ by \cite[Lemma 2.45]{kollar-mori}. Since $a(E,X,B+M) >0$ and the components of $\{B_W^{>0}\}$ are disjoint, we deduce that $V_0$ is a component of $ D \bigcap (\bigcap_i S_i)$ for some $D \subset \{B_W^{>0}\}$ and some $S_i \subset B_W^{=1}$.

		For the last assertion, let $R$ be the reduced divisor whose components are the components of $B_W$ passing through $V_0$. Since discrepancies are linear in the boundary,
		$$
		a(E,W,B_W)=a(E,W,R)+\sum_{G\subseteq R}(1-\mu_GB_W)\mu_EG,
		$$
		where $a(E,W,R)\in\Z_{\geq0}$ and every term of the sum is non-negative. As $a(E,W,B_W)=a(E,X,B+M)<1$, we have $a(E,W,R)=0$. Moreover, a component of $R$ with non-positive coefficient would contribute at least one to the sum, and $D$ is the only component of $R$ with coefficient in $(0,1)$. Hence the sum is $(1-b)\mu_ED$, and $b_E=1-(1-b)d$, where $d=\mu_ED\in\Z_{>0}$.
	\end{proof}
	
	Given a set $\Lambda\subset\R$ and a divisor $D$, we write $D\in\Lambda$ if every non-zero coefficient of $D$ belongs to $\Lambda$. Empty sums are understood to be zero.
	\begin{thm}\label{thm-finite-coeff}
		Let $\Sigma=\{n_1,\ldots,n_p\}$ be a finite set in $\Q$, and let $(X/Z,B+M)$ be a dlt g-pair such that $K_X+B+M=P+N$ where $N$ is a $\Q$-Cartier $\Q$-divisor whose support contains no lc centres of $(X,B+M)$. Then there exist finite sets $\Lambda,\Omega\subset(0,1]$ depending on $(X,B+M)$, $\Sigma$ and $P$ such that, for any birational map $\varphi: (X/Z,B+M) \bir (X'/Z,B'+M')$ to a dlt g-pair with the same data satisfying:
		\begin{enumerate}
			\item $K_{X'} +B'+M' =P'+N'$ with $^*P=^*P'$ and $^*N \ge ^*N'$ as b-divisors;
			
			\item $\varphi$ is an isomorphism between open subsets $U\subset X$ and $U' \subset X'$ intersecting all lc centres of $(X,B+M)$ and $(X',B'+M')$ respectively, with $^*(N|_{U}) =^*(N'|_{U'})$; and
			
			\item for any prime divisor $D \subset X'$, $\mu_D(N')$ is of the form $\sum_{i,j} n_{i,j}\frac{m_{i,j}'}{q_j'}$ where $n_{i,j} \in \Sigma$ and $m_{i,j}', q_j' \in \Z_{>0}$ satisfying $\frac{q_j'-1}{q_j'} \le \mu_D(B')$,
		\end{enumerate}
		we have $B' \in \Lambda$; and if there is a component $S \subset \rddown{B}$, then $B'_{S'} \in \Omega$. Moreover, for any prime divisor $V'$ of $S'$, $\mu_{V'}(N'|_{S'})$ is of the form $\sum_{i,j} n_{i,j}\frac{s_{i,j}}{r_j}$ where $s_{i,j}, r_j\in \Z_{>0}$ satisfying $\frac{r_j-1}{r_j} \le \mu_{V'}(B'_{S'})$.
	\end{thm}
	\begin{proof}
		Let $\pi:(W,B_W+M_W)\to X$ be a fixed log resolution of $(X,(B+N)+M)$ such that the components of $\{B_W^{>0}\}$ are disjoint, and let $h>0$ be an integer such that $hN$ is Cartier. Note that the set
		$$
		\Lambda_0=\{\,db_i-d+1\mid d\in\Z_{>0},\ 0<db_i-d+1<1\,\}
		$$
		is finite, where $b_i$ are the coefficients of $B_W^{>0}$.

		Let $D$ be a prime divisor on $X'$ with $0<b':=\mu_D B'<1$, and let $b=1-a(D,X,B+M)$. By (1) and (2), we have $b'\leq b<1$, because $b=1$ would imply $b'=1$. By Lemma \ref{lem-canonical-blow-up}, we see $b\in\Lambda_0$.
		Since ${}^*P={}^*P'$, we deduce
		$$
		0\leq b-b'=\mu_D({}^*N-{}^*N')<1.
		$$
		By the inequalities $(q_j'-1)/q_j'\leq b'\leq b$ in (3), the denominators $q_j'$ are bounded. Since $\mu_D({}^*N)\in\frac1h\Z$ and $\Sigma\subset\Q$ is finite, $b-b'$ belongs to a fixed discrete subgroup of $\Q$. It follows that $b'$ has only finitely many possible values, which together with $1$ form $\Lambda$.
		
		Now let $V'$ be a prime divisor of $S'$ with $0<c':=\mu_{V'}B'_{S'}<1$. 
		On a common log resolution dominating $W$, blowing up the prime divisor of the strict transform of $S'$ dominating $V'$, we get a prime divisor $E$ exceptional over $X$ with
		$$
		c'=1-a(E,X',B'+M')\leq c:=1-a(E,X,B+M)<1,\qquad c\in\Lambda_0,
		$$
		where the last two assertions follow from (2) and Lemma \ref{lem-canonical-blow-up}. Let $l$ be the local Cartier index at $V'$, and let $D_k$ be the components of $\Supp N'$ containing $V'$. We write $n'_k=\mu_{D_k}N'$ and $b'_k=\mu_{D_k}B'$. By Lemmas \ref{lem-index-co}--\ref{lem-coeff} and (3), we have
		$$
		\begin{gathered}
			c-c'=n-\sum_k n'_k\frac{e_k}{l},\qquad n\in\tfrac1h\Z,\quad e_k\in\Z_{>0},\\
			n'_k=\sum_{i,j}n_{k,i,j}\frac{m'_{k,i,j}}{q'_{k,j}},
		\end{gathered}
		$$
		and
		$$
		\frac{l-1}{l}\leq c'\leq c,\qquad
		\frac{q'_{k,j}-1}{q'_{k,j}}\leq b'_k
		\leq\frac{l-1+b'_k}{l}\leq c'.
		$$
		So there are only finitely many possibilities for $l$ and all $q'_{k,j}$. By the same lattice argument for $0\leq c-c'<1$, we deduce that $c'$ has only finitely many possible values. Together with $1$, these values form $\Omega$.
		
		For the last assertion, if $\mu_{V'}B'_{S'}=1$, then $V'$ is an lc centre. By assumption (2) and the hypothesis on $\Supp N$, we see that $N'$ vanishes at its generic point, so we can take the empty sum. Otherwise, with the local notation above, by Lemmas \ref{lem-index-co}--\ref{lem-coeff} we have
		$$
		\begin{gathered}
			\mu_{V'}(N'|_{S'})=\sum_{k,i,j}n_{k,i,j}\frac{m'_{k,i,j}e_k}{q'_{k,j}l},\\
			\frac{q'_{k,j}l-1}{q'_{k,j}l}
			\leq\frac{l-1+b'_k}{l}\leq\mu_{V'}B'_{S'}.
		\end{gathered}
		$$
		So we get the representation we need, also when $\mu_{V'}B'_{S'}=0$.
	\end{proof}
	
	We recall the following fact about discrepancies.
	
	\begin{lem}\label{lem-difficulty}
		Let $(X/Z,B+M)$ and $(X^+/Z,B^+ +M^+)$ be two lc g-pairs with birational underlying varieties and the same data. Let $f: X \to T$ and $f^+: X^+ \to T$ be two birational morphisms over $Z$. Suppose that
		\begin{itemize}
			\item $a(D,X,B+M) \le a(D,X^+,B^+ +M^+)$ for any prime divisor $D$ on $X^+$ (resp. on $X$);
			
			\item $K_X+B+M$ is anti-nef$/T$ (resp. anti-ample$/T$) and $K_{X^+}+B^+ +M^+$ is ample$/T$ (resp. nef$/T$); and
			
			\item $E$ is a prime divisor on $X^+$ (resp. on $X$) which is exceptional$/T$.
		\end{itemize}
		Then we have $a(E,X,B+M) < a(E,X^+,B^+ +M^+)$.
	\end{lem}
	\begin{proof}
		Let $\alpha:W \to X$ and $\beta: W \to X^+$ be a common resolution. The divisor $G:=\alpha^*(K_X+B+M) - \beta^*(K_{X^+}+B^+ +M^+)$ is anti-nef$/X^+$. By comparing discrepancies, $\beta_*G$ is effective, and hence so is $G$. The birational transform of $E$ is covered by curves over $T$ on which $G$ has negative intersection, so $\mu_E G >0$. The other case is analogous.
	\end{proof}
	
	\begin{thm}[A special termination]\label{thm-sp-term}
		Let $(X/Z,B+M)$ be a $\Q$-factorial dlt g-pair and $K_X+B+M=P+N$ where $P$ is a nef$/Z$ divisor and $N$ is a $\Q$-divisor. Given a $P$-trivial MMP on $K_X+B+M$ with scaling of some divisor $A$
		$$
		(X,B+M) \bir \cdots \bir (X^i,B^i+M^i) \bir (X^{i+1},B^{i+1}+M^{i+1}) \bir \cdots,
		$$
		let $S$ be an lc centre of $(X,B+M)$. Suppose that the above MMP does not contract $S$ and the induced map $(S^{i}, B_{S^{i}}+M_{S^{i}}) \bir (S^{i+1}, B_{S^{i+1}}+M_{S^{i+1}})$ is an isomorphism on $\rddown{B_{S^i}}$ for $i \gg 0$. Then after discarding finitely many steps, these maps lift to small $\Q$-factorialisations $(\widetilde{S}^{i}, B_{\widetilde{S}^{i}}+M_{\widetilde{S}^{i}}) \bir (\widetilde{S}^{i+1}, B_{\widetilde{S}^{i+1}}+M_{\widetilde{S}^{i+1}})$ forming a $P^i|_{\widetilde{S}^{i}}$-trivial MMP on $K_{\widetilde{S}^{i}}+B_{\widetilde{S}^{i}}+M_{\widetilde{S}^{i}}$ with scaling of $A^i|_{\widetilde{S}^{i}}$.
		
	\end{thm}
	\begin{proof}
		After reindexing, we can assume that the above MMP consists of flips and is an isomorphism at the generic points of lc centres. We then reduce to the case when $\Supp N$ contains no lc centres. Since $N$ is a $\Q$-divisor, we can write $mN$ as the difference of two base point free Cartier divisors for some $m>0$, and taking general members we obtain $N'\sim_\Q N$ whose support contains no lc centres. Replacing $(P,N)$ with $(P+N-N',N')$ does not change the MMP, the $P$-triviality and the scaling numbers, and ${}^*P^i$ is still independent of $i$, because $N-N'$ is a principal $\Q$-divisor.

		Applying Theorem \ref{thm-finite-coeff} inductively, we can find a finite set $\Omega \subset (0,1]$ such that $B_{S^i} \in \Omega$ for all $i$. Define the difficulty (cf. \cite[Definition 4.2.9]{fujino-sp})
		$$d(i)=d(S^i,B_{S^i}+M_{S^i}):=\sum_{b \in \Omega\cup\{0\}}\sharp \{E|a(E,S^i,B_{S^i}+M_{S^i}) < 1- b, c_{S^i}(E) \nsubseteq \rddown{B_{S^i}}\}.$$
		Clearly $d(i) < +\infty$ and $d(i+1) \le d(i)$ for all $i$. Let $S^i\to T^i\leftarrow S^{i+1}$ be the induced diagram. If $S^i \bir S^{i+1}$ extracts a prime divisor $E$, let $b=\mu_E B_{S^{i+1}}\in\Omega\cup\{0\}$. By Lemma \ref{lem-difficulty}, we have $a(E,S^i,B_{S^i}+M_{S^i})<1-b=a(E,S^{i+1},B_{S^{i+1}}+M_{S^{i+1}})$, so $E$ is no longer counted in the corresponding summand. Therefore, after reindexing, no divisors are extracted. By the same argument, the boundary coefficients of prime divisors that are not contracted are constant.
		
		Next we construct the lifting. Let $\widetilde S^i$ be a small $\Q$-factorialisation. Note that a small $\Q$-factorialisation of $S^{i+1}$ is a good minimal model$/T^i$ of $\widetilde S^i$: the boundary is the push-forward, the adjoint on $S^{i+1}$ is ample$/T^i$, and by Lemma \ref{lem-difficulty} the discrepancies of contracted prime divisors strictly increase. By Corollary \ref{cor-term}, there is a finite MMP$/T^i$. Its last model has a small morphism to $S^{i+1}$. Otherwise, there is a prime divisor on the last model which is exceptional over $S^{i+1}$, and by comparing its discrepancies on $\widetilde S^i$, on the last model and on $S^{i+1}$, we get a contradiction to the above strict increase. We use this model as $\widetilde S^{i+1}$ and continue. Joining these MMPs, we get an MMP with only finitely many divisorial contractions, because each divisorial contraction decreases the relative Picard number over $Z$. Hence $S^i\bir S^{i+1}$ is an isomorphism in codimension one for $i\gg0$.
		
		It remains to show that the lifted maps form a $P^i|_{\widetilde S^i}$-trivial MMP with scaling of $A^i|_{\widetilde S^i}$. Each of these MMPs$/T^i$ is $P^i|_{\widetilde S^i}$-trivial. Since $K_{\widetilde S^i}+B_{\widetilde S^i}+M_{\widetilde S^i}+\lambda_i A^i|_{\widetilde S^i}\equiv0/T^i$, it is also an MMP with scaling of the induced divisor $A^i|_{\widetilde S^i}$. Note that it is trivial for the restriction of the nef divisor $K_{X^i}+B^i+\lambda_iA^i+M^i+\alpha_iP^i$. Finally, the push-down of $A^i|_{\widetilde S^i}$ is $A^{i+1}|_{\widetilde S^{i+1}}$. This follows from the fact that the map $\widetilde S^i\bir\widetilde S^{i+1}$ is crepant for the g-pairs with boundaries $B_{\widetilde S^i}+\lambda_iA^i|_{\widetilde S^i}$ and $B_{\widetilde S^{i+1}}+\lambda_iA^{i+1}|_{\widetilde S^{i+1}}$.
		
	\end{proof}
	
	\begin{defn}[Degenerations]\label{defn-degen}
		Let $(X/Z,B+M)$ be a $\Q$-factorial dlt g-pair and $P$ be a nef$/Z$ divisor. Given a $P$-trivial MMP on $K_X+B+M$ with scaling of $A$ and scaling numbers $\lambda_i$, we say that it \emph{degenerates to an MMP on $K_X+B+M+\alpha P$} for some $\alpha\gg0$ if it is an MMP with scaling of $A$ and with the same scaling numbers. More precisely, $K_{X^i}+B^i+\lambda A^i+M^i+\alpha P^i$ is nef$/Z$ for every $i$ and $\lambda_{i+1}\leq\lambda\leq\lambda_i$. If the MMP ends at a $P$-minimal model $X^r$, we append $\lambda_{r+1}=0$ and include the final interval in this condition. In particular, $K_{X^r}+B^r+M^r+\alpha P^r$ must be nef$/Z$. If there are no steps, this is the only requirement. We do not append zero in the Mori fibre space case.
	\end{defn}
	
	\begin{rem}
		A $P$-trivial MMP which degenerates is not known to terminate in general. On the other hand, a $P$-trivial MMP which terminates does not necessarily degenerate, as Example \ref{exa-boundary-no-degeneration} shows.
	\end{rem}
	
	We give three criteria for degeneration. The first two correspond to Theorem \ref{thm-klt}, and the third allows more general effective divisors, with additional assumptions along $\rddown{B}$.
	
	\begin{thm}[\text{cf. Theorem \ref{thm-klt}}]\label{thm-degen}
		Let $(X/Z,B+M)$ be a $\Q$-factorial dlt g-pair, let $P$ be nef$/Z$, and write $N=K_X+B+M-P$. Given a $P$-trivial MMP on $K_X+B+M$ with scaling of $A$, suppose that one of the following holds:
		\begin{enumerate}
			\item $P\equiv\pm(tK_X-C-sN+D)$, where $t\in\Q_{\geq0}$, $s\geq0$, $sN$ is a $\Q$-Cartier $\Q$-divisor, $C,D\geq0$, $\Supp C\subseteq\Supp B$, and $\Supp((D-C-tB)^+)$ contains no lc centres of $(X,B+M)$; or
			\item there is a divisor $Q\equiv P$ such that $\Supp\{Q\}_{ir}\subseteq\Supp\{B\}$; or
			\item $N$ is a $\Q$-Cartier $\Q$-divisor and
			$$
			P\equiv\pm(tK_X-C+D-sN),
			$$
			where $t,s\in\Q_{\geq0}$, $C,D\geq0$ and $\Supp C\subseteq\Supp B$. Moreover, the given MMP terminates near $\rddown{B}$, and there is a number $\alpha_0\geq0$ such that
			$$
			(K_{X^i}+B^i+M^i+\alpha_0P^i)|_{\rddown{B^i}}
			$$
			is nef$/Z$ for all $i\gg0$.
		\end{enumerate}
		Then the given MMP degenerates to an MMP on $K_X+B+M+\alpha P$ for all sufficiently large $\alpha$.
	\end{thm}
	\begin{proof}
		Let $L=K_X+B+M$ and $L^i=K_{X^i}+B^i+M^i$. We prove (2) first, since (1) reduces to it.
		
		It suffices to find a constant $c\geq0$, independent of $i$, such that every $(L^i+cP^i)$-negative extremal ray is $P^i$-trivial for all $i>0$, including the final model if there is one. First consider the intervals with $\lambda_{i+1}>0$. Increasing $c$, we can assume $L^1+cP^1+\lambda_1A^1$ is nef. Suppose that some $L^i+cP^i+\lambda_{i+1}A^i$ is not nef, and let $X^i$ be the first such model. Then $L^i+cP^i+\lambda_iA^i$ is nef. If $\lambda_{i+1}\leq\lambda\leq\lambda_i$, then by convexity any $(L^i+cP^i+\lambda A^i)$-negative ray is also $(L^i+cP^i)$-negative, and hence it is $P^i$-trivial. On the other hand, by Definition \ref{defn-P-mmp}, $L^i+\lambda A^i+\beta P^i$ is nef for some $\beta\geq0$. This is a contradiction. So we have the required nefness on these intervals.

		If the MMP ends with a $P$-minimal model $X^r$, then by the same ray criterion and the $P^r$-nefness of $L^r$, we see that $L^r+cP^r$ is nef. So by convexity, the required nefness also holds on the final interval.
		
		\emph{Case (2).}
		
		As in Lemma \ref{lem-P-trivial}, there is a convex decomposition $Q=\sum_jr_jQ_j$ into $\Q$-divisors with positive $\Q$-linearly independent coefficients $r_j$, such that each $(X,B+Q_j-Q+M)$ is dlt. Each step is $Q_j^i$-trivial and is also an MMP step for $L^i+Q_j^i-Q^i$, so the auxiliary pairs remain lc. For an $L^i$-negative extremal curve $\Gamma$ of minimal length, by Lemma \ref{lem-lower-bnd} we have
		$$
		Q_j^i\cdot\Gamma
		=(L^i+Q_j^i-Q^i)\cdot\Gamma-L^i\cdot\Gamma+P^i\cdot\Gamma
		\geq-2\dim X.
		$$
		The Cartier indices of $Q_j^i$ are bounded as in Lemma \ref{lem-P-trivial}. Hence there is a fixed $\delta>0$ such that $P^i\cdot\Gamma>0$ implies $P^i\cdot\Gamma\geq\delta$. So we can apply the argument at the start of the proof with $c>2\dim X/\delta$.
		
		\emph{Case (1).}
		Recall the auxiliary pair and divisor in the proof of Theorem \ref{thm-klt}. We have
		$$
		L_\eta\equiv aL+bP,\qquad a>0,\qquad Q\equiv P,\qquad
		\Supp\{Q\}_{ir}\subseteq\Supp\{B_\eta\}.
		$$
		Since
		$$
		L_\eta+\lambda aA+\beta Q
		\equiv a(L+\lambda A+\alpha P),\qquad \beta=a\alpha-b,
		$$
		we can apply (2) with scaling divisor $aA$.
		
		\emph{Case (3).}
		If the sequence ends with a Mori fibre space, we only need to take $\alpha$ large enough. Otherwise, after discarding finitely many steps, we can assume that all steps are flips disjoint from $\rddown{B^i}$, and that the assumed nefness along the boundary holds on every model.

		Let $\sigma\in\{1,-1\}$ be such that $Q:=\sigma(tK_X-C+D-sN)\equiv P$. For a fixed $r\gg0$ and $0<\eta\ll1$, we consider the boundary
		$$
		B_\eta=\frac{B+\eta(D-C-r\rddown{B})}{1+\eta t},
		$$
		and we multiply the nef data by $(1+\eta t)^{-1}$. It is easy to see that the new g-pair is klt and $\Supp\{Q\}_{ir}\subseteq\Supp\{B_\eta\}$. In fact, every lc centre of $(X,B+M)$ is contained in $\rddown{B}$. So on a fixed log resolution, the pull-back of $\rddown{B}$ contains every lc place, and the coefficients of these places become less than one if $r\gg0$. Moreover, $\{Q\}_{ir}$ is supported on $C+D$. We denote by $L_\eta$ the adjoint divisor of the new g-pair. One calculates that
		$$
		L_\eta\equiv aL+bP-E,\qquad
		E=\frac{\eta r}{1+\eta t}\rddown{B}\geq0
		$$
		for some fixed real numbers $a>0$ and $b$.
		Since the given MMP is $P$-trivial and disjoint from $\Supp E$, it is also an $(L_\eta+\beta P)$-MMP for every $\beta\geq0$. Arguing as in Case (2) for $(X,B_\eta+M_\eta)$, we see that, for all sufficiently large $\beta$, every $(L_\eta^i+\beta P^i)$-negative extremal ray is $P^i$-trivial, uniformly in $i$.

		Now we pick $c\gg0$ with $c\geq\alpha_0$. Then any $(L^i+cP^i)$-negative extremal curve $\Gamma$ is not contained in $\rddown{B^i}$, so $E^i\cdot\Gamma\geq0$, and we have
		$$
		(L_\eta^i+(ac-b)P^i)\cdot\Gamma
		=a(L^i+cP^i)\cdot\Gamma-E^i\cdot\Gamma<0.
		$$
		So $P^i\cdot\Gamma=0$, and we conclude by the argument at the start of the proof.
		
	\end{proof}
	
	\begin{exa}[Termination near the boundary without degeneration]\label{exa-boundary-no-degeneration}
		Let $T,Q$ be as in Example \ref{exa-P-minimal-no-threshold}, and let $H$ be a very ample Cartier divisor on $T$. Using the quotient convention, consider
		$$
		\pi:X=\PP_T(\mathcal O_T\oplus\mathcal O_T(H))\to T,
		\qquad \xi=c_1(\mathcal O_X(1)).
		$$
		Let $S\sim\xi-\pi^*H$ be the section corresponding to the quotient onto $\mathcal O_T$. Then $\xi|_S=0$. Set
		$$
		B=S,\qquad P=\xi+\pi^*Q,\qquad M=2\xi+\pi^*Q.
		$$
		Note that $\xi$ is nef and big, and $(X,B+M)$ is a plt g-pair with data on $X$. Moreover, we have
		$$
		L:=K_X+B+M=P+\pi^*K_T.
		$$
		So $N:=L-P$ is Cartier. Since $P$ is nef and big, we have $P\equiv D\geq0$, and hence the numerical relation in condition (3) of Theorem \ref{thm-degen} holds if we take $t=s=0$ and $C=0$.

		The divisor $A=\xi+\pi^*H$ is ample. For every $\epsilon>0$, we see that
		$$
		L+\epsilon A+\alpha P
		=(1+\epsilon+\alpha)\xi+
		\pi^*\bigl(K_T+\epsilon H+(1+\alpha)Q\bigr)
		$$
		is nef for $\alpha\gg0$ (depending on $\epsilon$). Hence $L$ is $P$-nef. On the other hand, we have
		$$
		(L+\alpha P)|_S=K_T+(1+\alpha)Q,\qquad
		\bigl((L+\alpha P)|_S\bigr)^2=-1.
		$$
		In particular, $(L+\alpha P)|_S$ is not nef for any $\alpha\geq0$.
		
		To obtain a non-empty MMP, let $\mu:\widetilde X\to X$ be the blow-up of a point outside $S$ and $E$ the exceptional divisor. Let $\widetilde B=S^\sim$, $\widetilde M=\mu^*M$ and $\widetilde P=\mu^*P$. Then $\widetilde L:=K_{\widetilde X}+\widetilde B+\widetilde M=\mu^*L+2E$. Hence $\mu$ is $\widetilde P$-trivial and $\widetilde L$-negative, and it is disjoint from $\widetilde B$. We have seen that its target $X$ is $P$-minimal. Let $\widetilde A\equiv\mu^*A-\delta E$ be ample, where $0<\delta\ll1$. For $0<\lambda\ll1$, we pick $\alpha$ so that $L+\lambda A+\alpha P$ is ample. The divisor
		$$
		\widetilde L+\lambda\widetilde A+\alpha\widetilde P
		\equiv\mu^*(L+\lambda A+\alpha P)+(2-\lambda\delta)E
		$$
		has nef threshold $2/\delta-\lambda$ with respect to $\widetilde A$, and the corresponding contraction is $\mu$. Finally, if we represent $Q$ on general very ample divisors and take general effective scaling divisors with small coefficients, then the support conditions in Definition \ref{defn-P-mmp} hold and $(\widetilde X,\widetilde B+(2/\delta)\widetilde A+\widetilde M)$ is dlt.
		
		This illustrates the distinction between termination near $\rddown{B}$ and degeneration on $\rddown{B}$ in Theorem \ref{thm-degen}(3). Here $N=\pi^*K_T\sim-\pi^*C$ is not effective, so the example does not address the case $N\geq0$. Also $M|_S=Q$, so the data are not NQC.
	\end{exa}
	
	\subsection{Zariski decompositions}
	
	We apply the preceding induction to lc g-pairs with weak Zariski decompositions, generalising Corollary \ref{cor-klt'}.  \\
	\noindent \textbf{Weak Zariski decompositions.}
	Given a pseudo-effective divisor $D$, we say $D=P+N$ is a weak Zariski decomposition if $P$ is nef and $N \ge 0$. We propose the following conjecture.
	\begin{conj}[Restricted termination]\label{conj-term}
		Let $(X/Z,B+M)$ be a $\Q$-factorial dlt g-pair and $K_X+B+M=P+N$ be a weak Zariski decomposition such that $N$ is a $\Q$-Cartier $\Q$-divisor. Given a $P$-trivial MMP with scaling of $A$, suppose it degenerates to an MMP with scaling of $A$ on $K_X+B+M+\alpha P$ for some $\alpha>0$. Then it terminates with a minimal model of $(X/Z,B+M+\alpha P)$ after finitely many steps.
	\end{conj}
	This conjecture is much weaker than termination with ample scaling for arbitrary g-pairs, or equivalently, the existence of their minimal models (Corollary \ref{cor-term}). While it is unclear whether to expect the latter, degeneration imposes a strong condition on $P$, and hence on $M$ since $N$ is a $\Q$-Cartier $\Q$-divisor. This suggests that the conjecture should hold.
	
	\begin{lem}\label{lem-surface}
		Conjecture \ref{conj-term} holds in dimension $2$ and the degeneration assumption holds automatically in dimension $2$.
	\end{lem}
	\begin{proof}
		Since an MMP in dimension $2$ consists of only divisorial contractions, it terminates after finitely many steps with a $P$-minimal model $(Y,B_Y+M_Y)$. Because $K_Y+B_Y+M_Y$ is $P_Y$-nef and it has a Zariski decomposition $K_Y+B_Y+M_Y=Q+L$, we deduce $P_Y \cdot L_i >0$ for each $(K_Y+B_Y+M_Y)$-negative component $L_i$ of $L$. Hence $K_Y+B_Y+M_Y+\alpha P_Y$ is nef for some $\alpha >0$.
	\end{proof}
	
	\begin{lem}\label{thm-dlt}\label{cor-3-dim}
		Let $(X/Z,B+M)$ be a $\Q$-factorial dlt g-pair of dimension $d$, and let $K_X+B+M=P+N$ be a weak Zariski decomposition, where $N$ is a $\Q$-Cartier $\Q$-divisor whose support contains no lc centres. Assume Conjecture \ref{conj-term} holds in dimensions at most $d-1$ and either $P$ or the nef data $M'$ is log effective (Definition \ref{defn-log-eff}).
		Then any $P$-trivial MMP degenerates to an MMP on $K_X+B+M+\alpha P$ for all sufficiently large $\alpha$.
		
		In dimension three, it suffices instead to assume that either $P\equiv D\geq0$, or $M\equiv D-C$ with $D,C\geq0$ and $\Supp C\subseteq\Supp B$.
	\end{lem}
	\begin{proof}
		For each lc centre $S$, by adjunction we have
		$$
		K_S+B_S+M_S=P|_S+N|_S,
		$$
		where $N|_S\geq0$ is a $\Q$-Cartier $\Q$-divisor whose support contains no lc centres. Log effectivity is preserved under restriction and pullback to common models. By induction, Theorem \ref{thm-sp-term} and Conjecture \ref{conj-term}, the MMP terminates near $\rddown{B^i}$ and there is a constant $\alpha_0\geq0$ such that $(K_{X^i}+B^i+M^i+\alpha_0P^i)|_{\rddown{B^i}}$ is nef$/Z$ for all $i\gg0$. In dimension three, we use Lemma \ref{lem-surface} instead.
		
		Apply Theorem \ref{thm-degen}(3). If $P\equiv D\geq0$, let $t=s=0$ and $C=0$. Otherwise,
		$$
		P\equiv K_X-C+(B+D)-N
		$$
		is the required relation with $t=s=1$. For log effective nef data, we have $M\equiv D\geq0$ by push-down, so $C=0$.
	\end{proof}
	
	\begin{defn}[Log effectiveness]\label{defn-log-eff}
		Given a log smooth sub-dlt pair $(X',B')$ and a divisor $D'$, we say $D'$ is \emph{log effective (resp. log abundant)} if for any stratum $S'$ of $(X',B'^{=1})$ (including $X'$), we have $D'|_{S'} \equiv E' \ge 0$ for some $E'$ (resp. $D'|_{S'}$ is abundant).
	\end{defn}
	
	\begin{thm}\label{thm-dlt-abund}
		Assume termination of MMPs for usual dlt pairs in dimensions $\le d$.
		
		Let $(X/Z,B+M)$ be a $\Q$-factorial dlt g-pair of dimension $d$, $P$ be a nef$/Z$ divisor and $N:=K_X+B+M-P$ be a $\Q$-Cartier $\Q$-divisor. Suppose $M'$ is log abundant.
		Then any $P$-trivial MMP degenerates to an MMP on $K_X+B+M+\alpha P$ and terminates with a minimal model or a Mori fibre space of $(X/Z,B+M+\alpha P)$, for all sufficiently large $\alpha$.
	\end{thm}
	\begin{proof}
		If the MMP ends with a Mori fibre space $Y\to T$, then $P_Y\equiv_T0$ and the assertion is clear. Replacing $N$ and $P$ accordingly, we can assume that $\Supp N$ contains no lc centres. By induction and Theorem \ref{thm-sp-term}, after reindexing every step is a flip disjoint from $\rddown{B}$.
		
		By Lemma \ref{lem-nef-abundant-divisor}, there is a log resolution $f:W\to X$ with $M_W\sim_{\R,Z}g^*H$, where $g:W\to T$ and $H$ is nef and big$/Z$. Let $H\sim_{\R,Z}A_T+E_T$ with $A_T$ ample$/Z$ and $E_T\geq0$. We resolve $\Supp(B_W+g^*E_T)$, where $B_W=B^\sim+\Exc(f)_{\mathrm{red}}$. By Lemma \ref{lem-exc-2}, a dlt modification MMP$/X$ terminates with a crepant model $\pi:(Y,B_Y+M_Y)\to X$.
		
		Let $C_0$ be the part of $g^*E_T$ supported on $\rddown{B_W}$. For $0<\eta\ll1$, note that $H-\eta E_T\sim_{\R,Z}(1-\eta)H+\eta A_T$ is ample$/Z$. So we can pick a general effective representative $H_\eta\sim_{\R,Z}H-\eta E_T$ with small coefficients. Let
		$$
		C_W=\eta C_0,\qquad G_W=\eta(g^*E_T-C_0)+g^*H_\eta.
		$$
		Then $C_W+G_W\sim_{\R,Z}M_W$ and $(W,B_W+G_W)$ is dlt. Note that
		$$
		K_W+B_W+G_W\sim_{\R,Z}K_W+B_W+M_W-\eta C_0.
		$$
		So, after shrinking $\eta$, the above MMP is also a $(K_W+B_W+G_W)$-MMP. Hence $M_Y\sim_{\R,Z}C_Y+G_Y$, where $C_Y,G_Y\geq0$ and $\Supp C_Y\subseteq\rddown{B_Y}$. Moreover, $(Y,B_Y+G_Y)$ is dlt.

		Let $P_Y=\pi^*P$ and $N_Y=\pi^*N$. Note that $\pi$ is an isomorphism outside $\rddown{B}$ and the given flips are disjoint from $\rddown{B^i}$. So the flips lift to $Y$ directly, and the induced morphisms $\pi_i:Y^i\to X^i$ are crepant. Since the lifted flips are $C_Y$-trivial, they form an ordinary $(K_Y+B_Y+G_Y)$-MMP. By assumption, this MMP terminates.

		By abuse of notation, we denote the final model by $Y$ and the transformed divisors by the same symbols, and we let $L_Y=K_Y+B_Y+M_Y$. By Lemma \ref{lem-pullbacks}, $L_Y$ and its restrictions to lc centres are $P_Y$-nef. Since the nef data restricted to these lc centres are log abundant, by induction on small $\Q$-factorialisations of them we see that $(L_Y+\alpha_0P_Y)|_{\rddown{B_Y}}$ is nef$/Z$ for some $\alpha_0$. Theorem \ref{thm-degen}(3) also applies to $Y$ itself, where the MMP has no steps and degeneration just means that $L_Y+\alpha P_Y$ is nef$/Z$ (Definition \ref{defn-degen}). Taking $t=s=1$, $C=0$ and $D=B_Y+C_Y+G_Y$, we see that $L_Y+\alpha P_Y$ is nef$/Z$ for $\alpha\gg0$. Since the induced morphism from $Y$ to the final model of the original MMP is crepant, $L+\alpha P$ is also nef$/Z$ on that model. This completes the proof.
	\end{proof}
	
	\noindent \textbf{Nakayama--Zariski decompositions.}
	\begin{lem}\label{lem-degen}
		Let $(X,B+M)$ be a projective $\Q$-factorial dlt g-pair with data $M'$. Suppose that $P=P_\sigma(K_X+B+M)$ is nef. If a $P$-trivial MMP on $K_X+B+M$ degenerates to an MMP on $K_X+B+M+\alpha P$ for $\alpha \gg 0$, then it terminates with a minimal model.
	\end{lem}
	\begin{proof}
		Since the scaling numbers tend to zero, the MMP contracts $\Supp N_\sigma(K_X+B+M+\alpha P)=\Supp N_\sigma(K_X+B+M)$ after finitely many steps and hence reaches a model $(Y,B_Y+M_Y)$ on which $K_Y+B_Y+M_Y=P_Y$ is nef.
	\end{proof}
	
	We will use the following construction several times. Let $f:W\to X$ be a log resolution on which the data and the nef positive part descend. Let $B_W=B^\sim+\Exc(f)_{\mathrm{red}}$ and write
	$$
	K_W+B_W+M_W=f^*(K_X+B+M)+F=P_W+N_W,
	$$
	where $F\geq0$ is exceptional and $N_W=N_\sigma(f^*(K_X+B+M))+F$. We can assume $\Supp(B_W+N_W)$ is snc. Let $\widehat N\geq0$ be a rational divisor close to $N_W$ with the same support, whose coefficients are smaller along $B_W^{=1}$ and larger outside $\Supp B_W$. Then $\widehat B=B_W+\widehat N-N_W$ is a boundary whose exceptional coefficients are still positive, and $\Supp\widehat N$ contains no lc centres. Note that the remaining lc centres are also lc centres of the crepant pullback. So log effectivity or log abundance is preserved.

	We also have $N_\sigma(P_W+\widehat N)=\widehat N$. If a $P_W$-trivial MMP contracts $\Supp\widehat N=\Supp N_W$, then it also contracts $F$ and all the components where we changed the boundary. So its minimal model is a minimal model of the original g-pair.
	
	\begin{thm}\label{thm-dlt'}
		Assume Conjecture \ref{conj-term} holds in dimensions at most $d-1$.
		
		Let $(X,B+M)$ be a projective lc g-pair of dimension $d$ with data $M'$. Assume $K_X+B+M$ birationally has a Nakayama--Zariski decomposition with nef positive part. Suppose either $M'$ or the positive part is log effective (Definition \ref{defn-log-eff}). Then $(X,B+M)$ has a minimal model.
	\end{thm}
	\begin{proof}
		
		Let $(W,\widehat B+M_W)$ be as in the construction above. Lemma \ref{thm-dlt} implies degeneration, and then Lemma \ref{lem-degen} implies that this pair has a minimal model. Since the MMP contracts the unchanged support of the negative part, this is a minimal model of the original pair.
	\end{proof}
	
	In dimension three, Lemma \ref{thm-dlt} implies the following.
	
	\begin{cor}\label{cor-3-dim'}
		Let $(X,B+M)$ be a projective lc g-pair of dimension $3$. Assume $K_X+B+M$ birationally has a Nakayama--Zariski decomposition with nef positive part. Suppose either $K_X+B+M \equiv D \ge 0$ or $M \equiv D \ge 0$.
		Then $(X,B+M)$ has a minimal model.
	\end{cor}
	\begin{proof}
		We use the same construction on $f:W\to X$. If $K_X+B+M\equiv D\geq0$, then
		$$
		P_W\equiv f^*D-N_\sigma(f^*(K_X+B+M))\geq0.
		$$
		If $M\equiv D\geq0$, let $E_M=f^*M-M_W\geq0$. Then we have
		$$
		M_W\equiv f^*D-E_M,\qquad \Supp E_M\subseteq\Supp\widehat B.
		$$
		Here the inclusion holds because every exceptional divisor still has a positive coefficient in the boundary. Now the corollary follows from the three-dimensional assertion of Lemma \ref{thm-dlt} and Lemma \ref{lem-degen}.
	\end{proof}
	
	\begin{exa}[An example with non-zero negative part]\label{exa-non-nqc-application}
		Let $T,Q$ be the abelian surface and the nef divisor in Example \ref{exa: nonvanishing fail}, and let $H$ be a very ample divisor. We use the quotient convention and let $\pi:V=\PP_T(\mathcal O_T\oplus\mathcal O_T(H))\to T$ and $\xi=c_1(\mathcal O_V(1))$. Let $S\sim\xi-\pi^*H$ be the section with $\xi|_S=0$, and let
		$$
		B_V=S,\qquad M_V=2\xi+\pi^*Q,\qquad P_V=\xi+\pi^*Q.
		$$
		Then $K_V+B_V+M_V=P_V$. Note that $M_V$ and $P_V$ are nef and big. However, they are not NQC, because their restrictions to $S$ are equal to $Q$.

		Let $\mu:X\to V$ be the blow-up of a point outside $S$ and $E$ the exceptional divisor. Let $B=S^\sim$, $M=\mu^*M_V$ and $P=\mu^*P_V$, where the data is on $X$. Then $(X,B+M)$ is a plt g-pair and $L:=K_X+B+M=P+2E$. We have
		$$
		P_\sigma(L)=P,\qquad N_\sigma(L)=2E.
		$$
		The log canonical divisor is big. On the other hand, by restricting to $B\simeq T$, we see that neither $M$ nor $P$ is NQC. So Corollary \ref{cor-3-dim'} applies with non-zero negative part. In fact, the contraction $\mu$ is a $P$-trivial MMP step, and $(V,B_V+M_V)$ is the minimal model we obtain.
	\end{exa}
	
	We can apply Theorem \ref{thm-dlt-abund} to lc centres and prove the following result.
	\begin{thm}\label{thm-dlt-abund'}
		Assume termination of MMPs for usual dlt pairs in dimensions $\le d-1$.
		
		Let $(X,B+M)$ be a projective lc g-pair of dimension $d$. Assume $K_X+B+M$ birationally has a Nakayama--Zariski decomposition with nef positive part. Suppose $M'$ is log abundant.
		Then $(X,B+M)$ has a minimal model.
	\end{thm}
	\begin{proof}
		Taking a log resolution and perturbing the boundary as above, we can assume that $(X,B+M)$ is $\Q$-factorial dlt, $P=P_\sigma(K_X+B+M)$ is nef and $N=N_\sigma(K_X+B+M)$ is a $\Q$-Cartier $\Q$-divisor whose support contains no lc centres. We run a $P$-trivial MMP. Note that on each lc centre $S$, the adjoint divisor $P|_S+N|_S$ is pseudo-effective. By induction and Theorems \ref{thm-sp-term} and \ref{thm-dlt-abund}, we see that the MMP terminates near $\rddown{B^i}$, and there is a constant $\alpha_0\geq0$ such that $(K_{X^i}+B^i+M^i+\alpha_0P^i)|_{\rddown{B^i}}$ is nef for all $i\gg0$.
		
		As $M'$ is abundant, there is $E\sim_\R M$ with $E\geq0$. Then $P\equiv K_X+B+E-N$, so the MMP degenerates by Theorem \ref{thm-degen}(3), where we take $t=s=1$, $C=0$ and $D=B+E$. We conclude by Lemma \ref{lem-degen}.
	\end{proof}
	
	Because termination holds for usual dlt threefold pairs, we have the following corollary.
	\begin{cor}\label{cor-dlt-abund}
		Let $(X,B+M)$ be a projective lc g-pair of dimension $4$. Assume $K_X+B+M$ birationally has a Nakayama--Zariski decomposition with nef positive part. Suppose $M'$ is log abundant.
		Then $(X,B+M)$ has a minimal model.
	\end{cor}
	
	\noindent \textbf{Nakayama--Zariski decompositions and weak minimal models.}
	
	The following lemma uses the canonical blow-up of Lemma \ref{lem-canonical-blow-up}.
	\begin{lem}\label{induction-lem}
		
		Let $(X,B+M)$ be a projective lc g-pair such that $K_X+B+M$ birationally has a Nakayama--Zariski decomposition with nef positive part. Then there exists a log smooth dlt g-pair $(W,B_W+M_W)$ with Nakayama--Zariski decomposition
		$$
		L_W:=K_W+B_W+M_W=P_W+N_W
		$$
		such that:
		\begin{itemize}
			\item $P_W$ is nef, $N_W$ is a $\Q$-Cartier $\Q$-divisor, and $\Supp N_W$ contains no lc centres;
			\item $a(D,W,B_W+M_W)+\mu_DN_W>1$ for every prime divisor $D$ over $W$ with $\mu_DN_W>0$;
			\item a minimal model of $(W,B_W+M_W)$ is a weak minimal model of $(X,B+M)$.
		\end{itemize}
		Moreover, let $A\geq0$ be ample with $(W,B_W+A+M_W)$ lc, and fix $\alpha\geq0$. There is an $\epsilon_0>0$ such that, for every $0<\epsilon<\epsilon_0$ and every prime divisor $D$ over $W$ with $\sigma_D(L_W+\alpha P_W+\epsilon A)>0$,
		$$
		a(D,W,B_W+\epsilon A+M_W)+\sigma_D(L_W+\alpha P_W+\epsilon A)>1.
		$$
	\end{lem}
	\begin{proof}
		
		Replacing $(X,B+M)$, we can assume it is $\Q$-factorial dlt and $K_X+B+M=P+N$ is its Nakayama--Zariski decomposition with $P$ nef. Let $\pi:(W,\Delta_W+M_W)\to X$ be a log resolution of $(X,(B+N)+M)$ such that the components of $\{\Delta_W^{>0}\}$ are disjoint, where $K_W+\Delta_W+M_W=\pi^*(K_X+B+M)$. Write
		$$
		\Delta_W-\pi^*N=B_W-N_W,
		$$
		where $B_W,N_W\geq0$ have no common components. Then $K_W+B_W+M_W=\pi^*P+N_W$ is its Nakayama--Zariski decomposition, and a minimal model of $(W,B_W+M_W)$ is a weak minimal model of $(X,B+M)$. We increase the coefficients of $B_W$ and $N_W$ by the same sufficiently small amounts. Then $N_W$ becomes rational, and we still have $\mu_D(B_W-N_W)<0$ for every component $D$ of $N_W$. Note that the log smooth sub-lc boundary $\Theta:=B_W-N_W$ does not change, and it has negative coefficients along $\Supp N_W$. Hence, for every prime divisor $D$ over $W$ with $\mu_DN_W>0$, we have
		$$
		a(D,W,B_W+M_W)+\mu_DN_W=a(D,W,\Theta)>1.
		$$
		
		Now we prove the moreover part. For simplicity, we write $L=L_W$, $P=P_W$ and $N=N_W$. We have
		$$
		(1+\alpha)N=N_\sigma((1+\alpha)L)
		\leq N_\sigma(L+\alpha P)+\alpha N.
		$$
		Together with the nefness of $P$, this implies $N_\sigma(L+\alpha P)=N$. Let $a_D=a(D,W,B_W+M_W)$. We can assume that $N\neq0$. Let $\delta>0$ be the minimum of $-\mu_G\Theta$ over the components $G$ of $N$. Since $(W,\Theta+(1+\delta)\Supp N)$ is log smooth and sub-lc, we have $a_D+\mu_DN=a(D,W,\Theta)\geq1+\delta$ whenever $\mu_DN>0$. So we can pick $0<\tau<\delta/(1+\delta)$ such that
		$$
		(1-\tau)(a_D+\mu_DN)>1\qquad\text{whenever }\mu_DN>0.
		$$
		Let $H_\epsilon=L+\alpha P+\epsilon A$. By coefficientwise continuity, we have
		$$
		(1-\tau)N\leq N_\sigma(W,H_\epsilon)\leq N
		$$
		for all sufficiently small $0<\epsilon<\tau$.
		For every prime divisor $D$ over $W$, we have
		$$
		\mu_DN_\sigma(W,H_\epsilon)\leq\sigma_D(H_\epsilon)\leq\mu_DN,
		$$
		Here the upper bound holds because $H_\epsilon-N=(1+\alpha)P+\epsilon A$ is ample. Since $(W,B_W+A+M_W)$ is lc, $a(D,W,B_W+\epsilon A+M_W)\geq(1-\epsilon)a_D$. Therefore, whenever $\sigma_D(H_\epsilon)>0$,
		$$
		a(D,W,B_W+\epsilon A+M_W)+\sigma_D(H_\epsilon)
		\geq(1-\tau)(a_D+\mu_DN)>1.
		$$
	\end{proof}
	
	\begin{lem}[\text{cf.\cite[Lemma 2.4]{hashizumehu}}]\label{lem-adjunction}
		Let $(X,B+M)$ be a dlt g-pair and $S$ be an lc centre.
		Let $\pi \colon W\to X$ be a log resolution of $(X,B+M)$ which is an isomorphism over the generic points of lc centres
		and write $K_{W}+B_W+M_W=\pi^{*}(K_{X}+B+M)+E$, where $(W,B_W)$ is log smooth sub-lc and $E \ge 0$ is exceptional$/X$. If $E_i$ is a component of $E$ such that $\mu_{E_i} B_W - \mu_{E_i} E <0$, then $E_i|_{S_W}$ is exceptional over $S$, where $S_W$ is the birational transform of $S$.
	\end{lem}
	\begin{proof}
		
		Let $\Delta_W=B_W-E$ and define $\Delta_{S_W}$ by crepant adjunction:
		$$
		K_{S_W}+\Delta_{S_W}+M_{S_W}
		=\pi|_{S_W}^*(K_S+B_S+M_S).
		$$
		Since $(W,\Delta_W)$ is log smooth, $\mu_{E_i}\Delta_W<0$ implies $\mu_V\Delta_{S_W}<0$ for every component $V$ of $E_i|_{S_W}$. Therefore, $V$ is exceptional$/S$.
	\end{proof}
	
	\begin{thm}\label{thm-dlt''}
		Let $(X,B+M)$ be a projective lc g-pair with data $M'$. Assume $K_X+B+M$ birationally has a Nakayama--Zariski decomposition with nef positive part. Suppose either $M'$ or the positive part is log effective. Then $(X,B+M)$ has a weak minimal model.
	\end{thm}
	\begin{proof}
		
		\emph{Step 1.}
		By Proposition \ref{prop-wmm}, we can replace $(X,B+M)$ by the log smooth model of Lemma \ref{induction-lem}. We run a $P$-trivial MMP on $L=K_X+B+M=P+N$ with scaling of an ample divisor $A\geq0$ such that $(X,B+A+M)$ is dlt, where the scaling numbers $0<\lambda_i<1$ tend to zero (Definition \ref{defn-P-mmp}). For each $i>0$, we denote by $X^i$ the model appearing in the MMP on which $K_{X^i}+B^i+M^i+\lambda_iA^i+\alpha_iP^i$ is nef. If the MMP ends at a $P$-minimal model $X^r$, then we still pick $\lambda_i\downarrow0$ and let $X^i=X^r$ for all $i>r$. Note that the MMPs in Definition \ref{defn-P-mmp} have no further steps on a $P$-minimal model. We write $L_i=K_{X^i}+B^i+M^i$, and note that the pairs $(X^i,B^i+\lambda_iA^i+M^i)$ are dlt.
		
		\emph{Step 2.}
		We prove termination near lc centres by induction on their dimension. For an lc centre $S^i$, by induction and Theorem \ref{thm-sp-term}, after reindexing, the maps induced on these lc centres lift to a $P^i|_{\widetilde S^i}$-trivial MMP on small $\Q$-factorialisations $\widetilde S^i$ with scaling of $A^i|_{\widetilde S^i}$. Now we check the assumptions of Theorem \ref{thm-degen}(3). The termination and the nefness along the boundary hold by induction, and the numerical assumption holds because of the log effectivity. Hence there is a fixed $\alpha>0$ such that
		$$
		(L_i+\lambda_iA^i+\alpha P^i)|_{S^i}
		$$
		is nef for all $i\gg0$.
		
		For $i\gg0$, let
		$$
		H=L+\alpha P+\lambda_iA,\qquad
		D=L_i+\alpha P^i+\lambda_iA^i.
		$$
		By Lemma \ref{induction-lem} and coefficientwise continuity, we have $\Supp N_\sigma(H)=\Supp N$. Since $M,P$ are nef and $A$ is ample on $X$, by an ample perturbation we can find a big klt boundary $\Delta$ with $H\sim_\R K_X+\Delta$. Note that $X\bir X^i$ is $H$-non-positive. So by \cite{bchm} we can run a $D$-MMP with scaling of an ample divisor, and it terminates with a model $Y$ on which $D_Y$ is nef. Since $\mathbf B_+(H)\subseteq\Supp N$, these maps preserve the generic points of lc centres.
		
		\emph{Step 3.}
		We can write
		$$
		{}^*H={}^*D+\mathbf G={}^*D_Y+\mathbf E+\mathbf G,
		$$
		where $\mathbf E,\mathbf G\geq0$ are exceptional over $Y,X^i$, respectively, and $\mu_F(\mathbf E+\mathbf G)=\sigma_F(H)$ for every prime divisor $F$ over $X$. Let $\pi:W\to X^i$, $\rho:W\to Y$ be a log resolution on which these b-divisors descend. We can assume that it is an isomorphism over the generic points of lc centres on $Y$. Let $E=(\mathbf E)_W$ and let $S_W$ be the transform of $S^i$. For the crepant boundary $\Delta_W^i$ from $(X^i,B^i+\lambda_iA^i+M^i)$, by the moreover assertion of Lemma \ref{induction-lem} we have
		$$
		\mu_F(\Delta_W^i-E)<0
		$$
		for every component $F$ of $E$. Applying Lemma \ref{lem-adjunction} to $\rho$ with nef data induced by $M+\alpha P$, we see that $E|_{S_W}$ is exceptional$/S_Y$. By restriction, we have
		$$
		\pi|_{S_W}^*(D|_{S^i})=\rho|_{S_W}^*(D_Y|_{S_Y})+E|_{S_W}.
		$$
		The left-hand side is nef, which in turn implies that $E|_{S_W}=0$ by the negativity lemma.
		
		By the support equality, $N^i\leq c(\mathbf E)_{X^i}$ for some $c>0$. Since $\mathbf E$ is anti-nef over $X^i$, by the negativity lemma we have
		$$
		{}^*N^i\leq c\mathbf E.
		$$
		It follows that $\Supp N^i$ is disjoint from $S^i$, and $L_i|_{S^i}=P^i|_{S^i}$ is nef. Every $(P^i+N^i)$-negative curve lies in $\Supp N^i$, so all subsequent steps avoid $S^i$.
		
		\emph{Step 4.}
		By induction, the MMP terminates near $\rddown{B^i}$ and $L_i|_{\rddown{B^i}}$ is nef. We conclude by Theorem \ref{thm-degen}(3) and Lemma \ref{lem-degen}.
	\end{proof}

\end{document}